\documentclass[11pt,a4paper]{article}

\usepackage{amssymb}
\usepackage{array}
\usepackage[utf8]{inputenc}
\usepackage[T1]{fontenc}
\usepackage{url,hyperref}

\usepackage[a4paper,left=2.3cm,right=2.3cm,top=2.8cm,bottom=2.8cm]{geometry}
\usepackage[dvips]{graphicx}
\usepackage{amsfonts}
\usepackage{amssymb}
\usepackage{amsmath}
\usepackage{color,url}
\usepackage{diagbox,booktabs,amsmath}
\usepackage[table]{xcolor}
\usepackage[export]{adjustbox}
            \usepackage[math]{cellspace}
    \usepackage{algorithmicx}
    \usepackage[ruled]{algorithm}
\usepackage[noend]{algpseudocode}
\usepackage{wrapfig}

\def\CC{{C\hspace{-.05em}\raisebox{.4ex}{\tiny\bf ++} }}

\newcommand{\off}[1]{}
\newtheorem{proposition}{Proposition}
\newenvironment{proof}{\emph{Proof.}}{$\Box$}
\newtheorem{remark}{Remark}
\DeclareMathOperator{\argmin}{argmin}
\newcommand{\uargmin}[1]{\underset{#1}{\argmin}\;}

\newcommand{\Un}{\mathbf{1}}
\newcommand{\R}{\mathbb{R}}
\newcommand{\norm}[1]{\|{#1}\|}
\newcommand{\dime}{s}
\def\IR{\relax{\rm I\kern-.18em R}}

\def\<{\left\langle}
\def\>{\right\rangle}

\DeclareMathOperator{\diag}{diag}

\begin{document}

\title{Approximation of Wasserstein distance with Transshipment}
\date{}
\author{Nicolas Papadakis\thanks{CNRS, Institut de Math\'ematiques de Bordeaux, 33405 Talence,  France}}
\maketitle

\begin{abstract}
An algorithm for approximating the $p$-Wasserstein distance  between histograms defined on unstructured discrete grids is presented.  It is based on the computation of a barycenter constrained to be supported on a low dimensional subspace, which corresponds to a transshipment problem. A multi-scale strategy is also considered. The method provides sparse transport matrices and can be applied to large  and non structured data.

\end{abstract}

\section{Introduction}
The computation of optimal transportation between two discrete normalized mesaures $\mu_x$ and $\mu_y$  defined from $x\in\Omega\subset \R^d$ to $[0;1]$ remains a challenging problem when an accurate discretization of the  domain $\Omega$ is  considered. 
Optimal transportation requires to define a ground distance between points $x,y\in\R^d$ of the domain $\Omega$. This ground metric is then used to measure how much it costs to move $\mu_x(x)$ to $\mu_y(y)$. Ground distances  $||x-y||^p$ are here considered for $p\geq 1$, which leads to the $p-$Wasserstein distance between $\mu_x$ and $\mu_y$ (see \cite{villani2008optimal,santambrogio2015optimal,2018arXiv180300567P} for  more detailed introductions). Such distances give robust metrics in retrieval applications for $1\leq p<2$  \cite{Rubner98,Pele-eccv2008,Pele-ICCV,hurtut2008}.
The underlying sparse transport matrix is also a useful tool for interpolation and transfer purposes \cite{Rabin11,Ferradans2014,2015-solomon-siggraph}.

 \paragraph{Computing Wasserstein distances} 
 The computation of Wasserstein distances is only explicit for $d=1$. A standard approach to estimate Wasserstein distances when $d>1$ consists in  pre-computing a  cost matrix  $||x-y||^p$ for $x,y\in\Omega$ and then estimating the whole transport matrix which dimension grows quadratically with $n=|\Omega|$. Linear programming or transportation simplex \cite{Luenberger2015,journal.pone.0110214} can be applied to estimate $p-$Wasserstein distance for $p\geq 1$ but they are  limited in practice to low dimensional problems, i.e. small values of $n$, for complexity  and storage issues.  The implementation provided in \cite{Bonneel} nevertheless allows to tackle problems of interesting  dimensions, i.e. $\Omega$ discretized with  more than $n=10^4$ points.
By exploiting the sparsity of the transport map, multi-scale strategy \cite{2015arXiv150903668O} or  grid refinement \cite{Schmitzer2016} can deal with larger problems by solving iteratively sparse low dimensional problems with linear programming.  More efficient algorithms can be considered in the specific case of $1-$Wasserstein distances \cite{Li2018}.

\paragraph{Approximation of Wasserstein distance} For large scale problems,  approximated sparse transport matrices and  Wasserstein distances  can be obtained by considering successive one dimensional problems with the  so-called sliced Wasserstein distance \cite{Rabin11,Bonneel2015}.
The  entropic regularization of the transport map proposed in \cite{cuturi13} is another relevant way to deal with problems of high dimension. Given a regularization parameter $\gamma>0$,  it approximates the true Wasserstein distance (that corresponds in this setting to $\gamma=0$) with the well-known Sinkhorn algorithm. The estimated transport matrices are nevertheless dense and should be truncated with care for interpolation purposes.
When data are discretized on an uniform grid, the estimation of the distance can  be obtained through iterative convolutions \cite{2015-solomon-siggraph} with a kernel $K_p=\exp(-||x-y||^p/\gamma)$.
This leads to very fast algorithm for the $p-$Wasserstein distance as the Kernel is separable when considering the $L^p$ norm $||.||_p$, and only convolutions and storage of $d$ one-dimensional   kernels are necessary. For small values of  $\gamma$, numerical instabilities nevertheless arise and dedicated attention must be given to the implementation   by considering for instance 
decaying values of $\gamma$, stabilization in log domain or multi-resolution approaches \cite{Schmitzer16}. 
The numerical convergence is also reduced with low values of $\gamma$, and overrelaxation \cite{Schmitz,Thibault} or greedy coordinate descents \cite{greenkhorn}  approaches have been proposed to tackle this issue.
When the kernel $K_p$ is not separable, it must be carefully truncated to zero to store sparse matrices. If not considering more complex and adaptive truncations \cite{Schmitzer16}, it limits the possible amplitude of the transport, which may be annoying in case of large displacements.
Other regularizations can then be of interest \cite{JMLRrot,2017arXiv171006276B,2017arXiv171102283S}, namely  when  it is suitable to recover sparse transport matrices.

However, as mentioned before, this kind of techniques only leads to efficient implementations for   densities discretized on structured grids. Moreover, if large displacements are involved in the data, the Wasserstein distance can not be accurately approximated, since numerical instabilities  arises with $||x-y||^p/\gamma$ when $\gamma$ goes to $0$.

\paragraph{Wasserstein barycenters and Transshipment} The $p-$Wasserstein barycenter, as introduced in \cite{agueh2011barycenters},  
can be used to approximate the $p-$Wasserstein distance between measures involving a large number $n$ of dirac masses.  
As proposed  in \cite{Ye2017} for clustering problems, the discrete barycenter between $2$ discrete measures can be parameterized with  a weighted sum of $\kappa<<n$ diracs to obtain a low  dimensional problem that corresponds to the transshipment problem of $n$ resources with $\kappa$ intermediate locations. The sum of the distances between each data and the barycenter then gives an approximation of the effective distance, as it has been proposed for $1-$Wasserstein distances \cite{2018arXiv180507416A}. Statistical properties of such a method have been later studied in \cite{Weed,2019arXiv190108949P}, where it has been underlined the robustness of this low rank regularization of the transport matrix to data outliers.  These ideas  have been extended   to discrete  approximation of barycenter between continuous measures  in \cite{2018arXiv180205757C}. 

\paragraph{Content}The use of low dimensional barycenters is the point of view adopted in this note to propose fast approximation of Wasserstein distances involving sparse transport matrices. 
As in \cite{Ye2017}, at each iteration of the presented algorithm, the $\kappa$ locations  of the barycenter are updated and  a linear program of dimension $2\kappa n$ is solved. Compared to the dimension $n^2$ for the classical Wasserstein distance, the overall complexity of the algorithm becomes attractive for high dimensional data.
In \cite{Ye2017} and  \cite{Weed,2019arXiv190108949P}, this low dimensional barycenter problem is respectively solved with the Alternating Direction Method of Multipliers  or Sinkhorn iterations \cite{CuD14}. 
The barycenter problem between two discrete densities is here seen as a transshipment problem and solved up to numerical accuracy with an efficient network simplex graph algorithm based on the work of \cite{Bonneel}. Then $\kappa$ transportation subproblems are solved in parallel to recover a sparse transport matrix.
Following \cite{Schmitzer16,2015arXiv150903668O,2018arXiv181000118L}, this leads to the design of a  multi-scale barycenter scheme, to iteratively refine the transport matrix and the associated approximation of the Wasserstein distance. 

The theoretical computer science community has recently provided improved bounds  for graph based algorithms solving approximate transportation problems up to $\epsilon$ additive or $1+\epsilon$ multiplicative errors. This namely includes the optimal transportation problem \cite{2018arXiv181007717B}, its entropic regularization  \cite{2018arXiv181010046A} or the transshipment problem \cite{2016arXiv160705127B}. As for the sliced method \cite{Rabin11}, the presented approach does not have such theoretical guarantees but very good performances are observed  in practice.
More precisely, approximate $p-$Wasserstein distances between cloud points of $10^5$ elements are obtained in a few minutes without involving prohibitive memory storage  issues. The proposed empirical algorithm can also be directly applied to non structured data. Thanks to the multi-scale refinement approach, a sparse transport map is  provided which can be of interest for interpolation purposes.

\newpage

\paragraph{Outline}
Section 2 details how to approximate  the Wasserstein distance when computing the barycenter between two discrete measures.
The strategy of \cite{Ye2017} is recalled in the general context of $p-$Wasserstein distance.
A multi-scale algorithm for recovering sparse transport matrices is finally presented.  The  performances of the algorithm are discussed in section 3 through extensive experiments realized on the  benchmark \cite{bench}. Numerical results show that this whole empirical process is efficient, namely when one of the two input data is spatially regular.

\section{Approximate Wasserstein distance from barycenter estimation}
Let  $\mu_x$ and $\mu_y$ be two  discrete measures defined on $\Omega\subset \R^d$, $d\geq 1$: $\mu_x=\sum_{i=1}^mw^x_i\delta_{x_i}$ and $\mu_y=\sum_{j=1}^nw^y_j\delta_{y_j}$.
 These measures are supported at positions $\{x_i\}_{i=1}^m$ and $\{y_j\}_{j=1}^n$, $x_i,y_j\in \Omega\subset \R^d$. They have normalized positive weights vectors $w^x\in \mathcal{S}_m$ and $w^y\in \mathcal{S}_n$, where $\mathcal{S}_n$ is  the simplex of size $n$ defined as $\mathcal{S}_n=\{w\in \R^n_+,\textrm{ s.t } \sum_{i=1}^n w_i=1\}.$ 
For $p\geq 1$, let $c^{xy}\in\R_+^{m\times n}$ be the ground cost matrix over  $\Omega$ defined as $c^{xy}_{ij}=\|x_i-y_j\|^p
$, which corresponds to the power $p$ of the distance related to a given  norm on $\R^d$. 
Then the $p-$Wasserstein distance between discrete measures $\mu_x$ and $\mu_y$ is 
\begin{equation}\label{def:Wp_discrete}W^p_p(\mu_x,\mu_y)=\min_{\gamma\in\mathcal{P}(w^x,w^y) } \langle \gamma,c^{xy}\rangle:=\sum_{ij}\gamma_{ij}c^{xy}_{ij},\end{equation}
with the set of admissible transport matrices 
\begin{equation}\label{admiss}\mathcal{P}(w^x,w^y)=\{\gamma\in\R_+^{m\times n},\, \textrm{ s.t }\gamma1_n=w^x,\gamma^\top 1_m=w^y\},\end{equation}
and where $1_n$ the  vector full of ones in $\R^n$.
This problem can be efficiently solved with  linear programming. It can also been formulated through a directed graph containing $m+n$ nodes and $mn$ vertices. The final transport matrix $\gamma$ is very sparse in practice (at most $m+n-1$ non null entries) but the involved complexity and memory storage scale with the product of data dimensions $nm$. For latter purpose, Algorithm \ref{algoW} details the function estimating the distance $W_p^p(\mu_x,\mu_y)$.

\begin{algorithm}[ht!]
\caption{Estimate $p$-Wasserstein distance $W_p^p(\mu_x,\mu_y)$}
\label{algoW}
\begin{algorithmic}[1]
\Procedure{Wp}{$x$, $w^x$, $y$, $w^y$, $p$}
\State Set $c^{xy}_{ij}=\norm{x_i-y_j}^p$
\State Solve problem \eqref{def:Wp_discrete} under the constraints \eqref{admiss} to get $\gamma$
\State Set $W=\langle \gamma,c^{xy}\rangle$
\State \Return$W$, $\gamma$
\EndProcedure
\end{algorithmic}
\end{algorithm}

\subsection{Interpolation and barycenters}
Let $\gamma^{xy}$ be an optimal transport matrix  solution of \eqref{def:Wp_discrete}
and $\mu_t$ be the following interpolation between measures $\mu_x$ and $\mu_y$ for $t\in[0;1]$ :

\begin{equation}\label{def:interp}
\mu_t=\sum_{ij}\gamma^{xy}_{ij}\delta_{x_i+t(y_j-x_i)}:=\sum_{k}w^t_k\delta_{z_k}.
\end{equation}
This interpolation is the discrete analogue \cite{2018arXiv180300567P} to the geodesic between $\mu_x$ and $\mu_y$ defined by the McCann’s interpolation \cite{MCCANN1997153}.
Discrete versions of some results in  \cite{santambrogio2015optimal} can now be expressed.

\begin{proposition}For $p\geq 1$, $\gamma^{xy}$ a solution of \eqref{def:Wp_discrete}, $\mu_t$ defined in \eqref{def:interp} and $t\in[0;1]$, the following relations hold:
\begin{equation}\label{constant_speed}W_p(\mu_x,\mu_t)=tW_p(\mu_x,\mu_y)\hspace{1cm}W_p(\mu_y,\mu_t)=(1-t)W_p(\mu_x,\mu_y),\end{equation} so that
\begin{equation}\label{pb:bar}W_p(\mu_x,\mu_y)=W_p(\mu_x,\mu_t)+W_p(\mu_y,\mu_t).
\end{equation}
\end{proposition}
\begin{proof}
Observing that $$W^p_p(\mu_x,\mu_t)=\min_{\gamma\in\mathcal{P}(w^x,w^t) }\sum_{ik}\gamma_{ik}\|x_i-z_k\|^p\leq \sum_{ij}\gamma^{xy}_{ij}\|t(x_i-y_j)\|^p,$$
leads to the upper bound  $W_p(\mu_x,\mu_t)\leq t W_p(\mu_x,\mu_y)$,   $\forall t\in[0;1]$. It can be shown  in  the same way  that  $W_p(\mu_y,\mu_t)\leq (1-t) W_p(\mu_x,\mu_y)$. 
Since $W_p$ is a distance, the triangle inequality  $W_p(\mu_x,\mu_y)\leq W_p(\mu_x,\mu_t)+W_p(\mu_t,\mu_y)$ involves that the previous relations are in fact equalities.
\end{proof}

\vspace*{0.4cm}

Following \cite{agueh2011barycenters}, it can be shown that the mid interpolation  $\mu_{1/2}$ is solution of the p-Wasserstein barycenter problem   between $\mu_x$ and $\mu_y$ with weights $(1/2,1/2)$.

 \begin{proposition}\label{prop:bar}
For $p\geq 1$  and the  interpolation $\mu_{1/2}$ defined in \eqref{def:interp},
it holds that
\begin{equation}\label{min:bary2}\frac12W^p_p(\mu_x,\mu_{1/2})+ \frac12W^p_p(\mu_y,\mu_{1/2})=\left(\frac12 W_p(\mu_x,\mu_y)\right)^p\end{equation}
and  $\mu_{1/2}$
 is a solution of the $p-$Wasserstein barycenter problem:
\begin{equation}\label{pb:bary2}\mu_{1/2}\in \uargmin{\mu} \frac12W^p_p(\mu_x,\mu)+ \frac12W^p_p(\mu_y,\mu).
\end{equation}
\end{proposition}
\begin{proof}
From \eqref{constant_speed}, it can first be noticed that: 
\begin{equation}\label{relation1}
\begin{split}
\frac12\left(W^p_p(\mu_x,\mu_{1/2})+W^p_p (\mu_y,\mu_{1/2})\right)=\frac12\left(2 \left(\frac12 W_p(\mu_x,\mu_y)\right)^p  \right)=\left(\frac12 W_p(\mu_x,\mu_y)\right)^p.
\end{split}
\end{equation}
Since  $W_p$ is a distance and  the function $|.|^p$, is convex for $p\geq 1$, it can next be observed that  $\forall\mu$:
\begin{equation}\label{relations2}\left(\frac12 W_p(\mu_x,\mu_y)\right)^p\leq \left(\frac12(W_p(\mu_x, \mu)+W_p(\mu_y, \mu))\right)^p\leq \frac12(W^p_p(\mu_x, \mu)+  W^p_p(\mu_y,\mu)),\end{equation}
Combing relations  \eqref{relation1} and \eqref{relations2} implies that $\mu_{1/2}$ is a solution of the barycenter problem \eqref{pb:bary2}.
\end{proof}
\vspace*{0.5cm}

Existence (and uniqueness  for $p> 1$) of Wasserstein  barycenters have been deeply studied in \cite{agueh2011barycenters,gouic2015existence}. 
As stated in the following proposition, the $p-$Wasserstein distance can be obtained  through the resolution of the barycenter problem \eqref{pb:bary2}.

\begin{proposition}
Let $\tilde \mu$ be a solution of the $p-$Wasserstein barycenter problem
\eqref{pb:bary2}, then  
\begin{equation}\label{rel:bar}
W_p(\mu_x,\mu_y)=W_p(\mu_x,\tilde  \mu)+W_p(\mu_y,\tilde \mu)
\end{equation}
and $W_p(\mu_x,\tilde  \mu)=W_p(\mu_y,\tilde \mu)$ for $p>1$.
\end{proposition}
\begin{proof}
First assume that \eqref{rel:bar} is not satisfied then we get a contradiction since 

\begin{equation}\label{relations2b}\left(\frac12 W_p(\mu_x,\mu_y)\right)^p\hspace{-0.1cm}< \left(\frac12(W_p(\mu_x,\tilde \mu)+W_p(\mu_y,\tilde \mu))\right)^p\hspace{-0.1cm}\leq \frac12(W^p_p(\mu_x,\tilde \mu)+  W^p_p(\mu_y,\tilde \mu))=\left(\frac12 W_p(\mu_x,\mu_y)\right)^p\hspace{-0.1cm},\end{equation}
where the last equality comes from Proposition  \ref{prop:bar} and the fact that   $\tilde \mu$ is a solution of \eqref{pb:bary2}.
Without loss of generality,  assume that $W_p(\mu_y,\tilde \mu)=\alpha W_p(\mu_x,\tilde \mu)$, 
with $\alpha\in [0;1]$. From Proposition  \ref{prop:bar}, it holds  that  
 $2^{p-1}(1+\alpha^p)W^p_p(\mu_x,\tilde \mu)=W^p_p(\mu_x,\mu_y)$. 
Relation \eqref{rel:bar} then gives  $W_p(\mu_x,\tilde  \mu)+W_p(\mu_y,\tilde \mu)=(1+\alpha) W_p(\mu_x,\tilde \mu)=W_p(\mu_x,\mu_y)$, which leads to  $((1+\alpha)/2)^p=(1+\alpha^p)/2$. As the function $|x|^p$ is strictly convex for $p>1$, we get $\alpha=1$ as soon as $p>1$.
\end{proof}

\newpage

\subsection{Barycenter computation}
A discrete barycenter $\mu_{1/2}$ between $\mu_x$ and $\mu_y$ can be obtained by solving \eqref{pb:bary2} with the distance \eqref{def:Wp_discrete}. As done  in \cite{CuD14} with an additional entropic regularization, the mid barycenter $\mu_{1/2}$ between $\mu_x$ and $\mu_y$ can be constrained to be supported on a set of $\kappa$ dirac masses, i.e. $\mu_{1/2}= \sum_{k=1}^\kappa w^z_k\delta_{z_k}$, with positions $z_k\in \R^d$ and  weights $w^z\in \mathcal{S}_\kappa$ . The barycenter problem can be rewritten as:
\begin{equation}\label{pb:bar_reg}
\begin{split}&\min_{\begin{array}{c}\{z_k\}_{k=1}^\kappa\in \Omega^\kappa\\w^z\in \mathcal{S}_\kappa
\end{array}}\hspace{-0.2cm} W^p_p\left(\mu_x,\mu_z\right)+W^p_p \left(\mu_y,\mu_z\right)=\hspace{-0.2cm} \min_{\begin{array}{c}\{z_k\}_{k=1}^\kappa\in \Omega^\kappa\\
(\gamma^x,\gamma^y)\in \mathcal{\tilde  P}(w^x,w^y)\end{array}}\hspace{-0.4cm} \langle \gamma^x,c^{xz}\rangle+ \langle \gamma^y,c^{yz}\rangle,\end{split}\end{equation}
with the admissible set of matrices
$$\mathcal{\tilde  P}(w^x,w^y)=\{\gamma^x\in\R_+^{m\times \kappa} ,\, \gamma^y\in\R_+^{n\times \kappa},\textrm{ s.t. } \gamma^x1_\kappa=w^x, \gamma^y1_\kappa=w^y, (\gamma^x)^\top1_m-(\gamma^y)^\top1_n=0_\kappa\}$$
 and the cost matrices 
$c^{xz}_{ik}=\|x_i-z_k\|^p$ 
 and $c^{yz}_{jk}=\|y_j-z_k\|^p$. 
Notice that the weight vector of the barycenter  $w^z\in\R^\kappa$ is implicitly included in the set of constraints: $(\gamma^x)^\top1_m=(\gamma^y)^\top1_n=w^z$. The positions $z_k$ here act as intermediate locations where the mass has to transit from $x$ to $y$.
Following \cite{Ye2017} and as illustrated in Figure \ref{fig:ts},  the idea behind this modeling  is to consider a limited number of transshipment locations $\kappa$ to speed up the computation.  \\

\begin{figure}[ht!]
\begin{center}
\begin{tabular}{cc}
\includegraphics[width=0.32\linewidth]{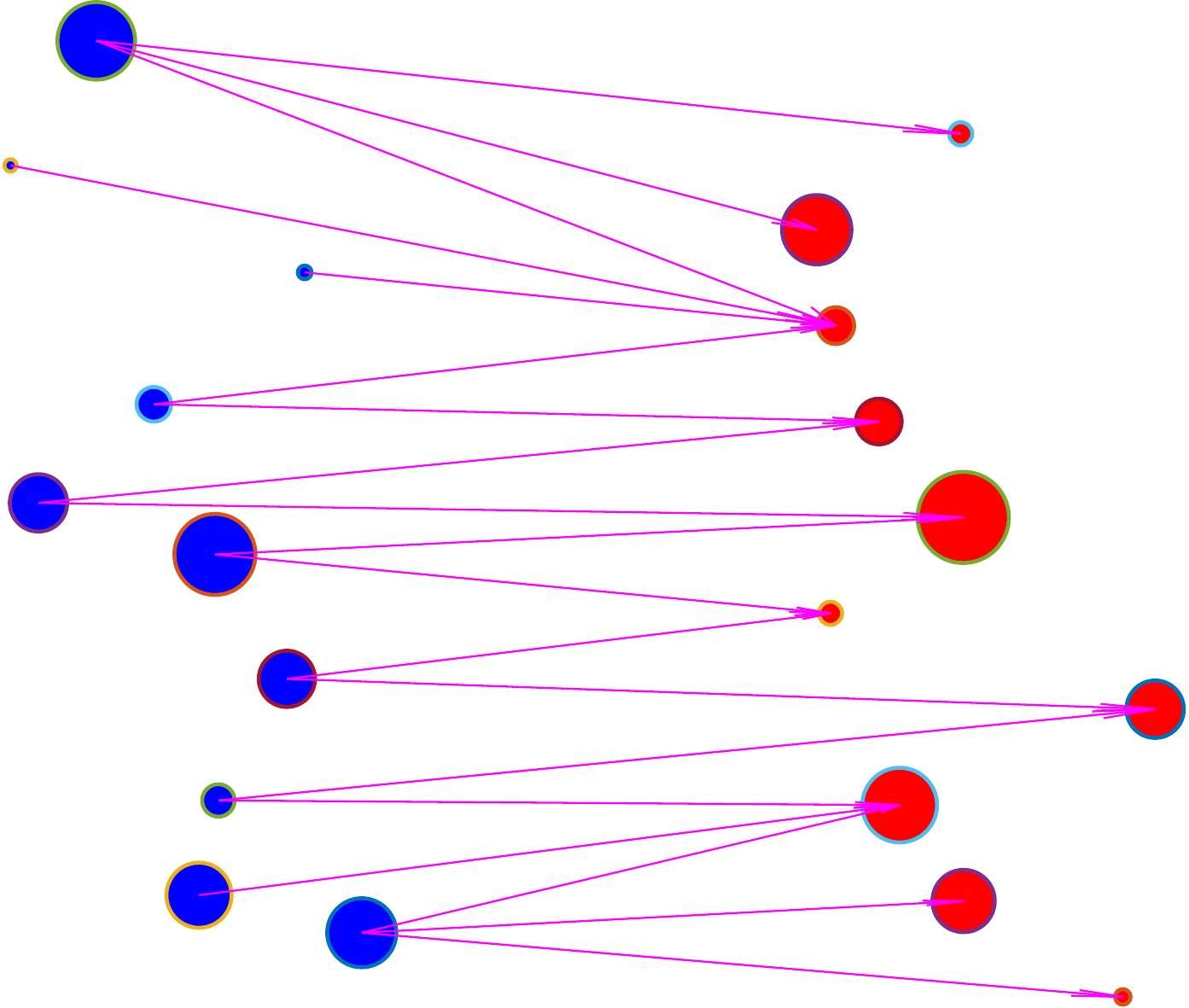}&\includegraphics[width=0.32\linewidth]{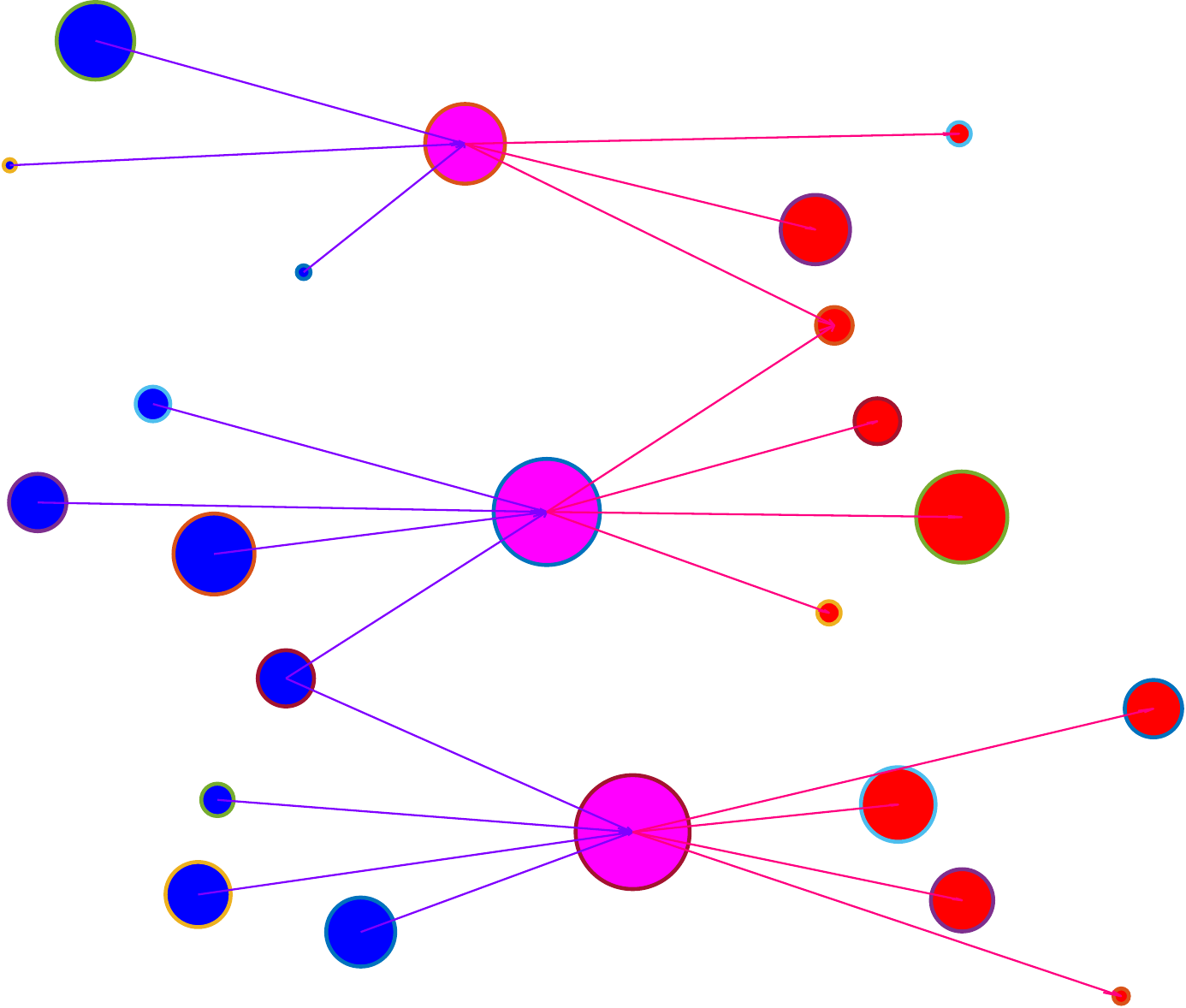}\\
Optimal Transportation problem&Optimal Transshipment problem
\end{tabular}
\caption{\label{fig:ts} Illustration of transshipment with $\kappa=3$ intermediate locations.}
\end{center}
\end{figure}

The problem \eqref{pb:bar_reg} is separately convex with respect its variables $\gamma^x$, $\gamma^y$ and  $z_k$.  For $p>1$,  the coupling terms are differentiable and  alternate minimization over transport matrices $(\gamma^x,\gamma^y)$  and  dirac positions $z_k$ converges  \cite{Tseng} to a saddle point.
Following \eqref{rel:bar}, the Wasserstein distance $W_p(\mu_x,\mu_y)$ can  be approximated with $\tilde W_p(\mu_x,\mu_y)$ obtained from transport  matrices $\gamma^x$ and $\gamma^y$ solutions of the problem \eqref{pb:bar_reg} as
$$W_p(\mu_x,\mu_y)\leq\tilde W_p(\mu_x,\mu_y)= W_p(\mu_x,\mu_{1/2})+W_p(\mu_y,\mu_{1/2})= \langle \gamma^x,c^{xz}\rangle^{1/p}+ \langle \gamma^y ,c^{yz}\rangle^{1/p}.$$
The alternate optimization steps for solving  \eqref{pb:bar_reg}  are now detailed and the process is summed up  in Algorithm \ref{algo_bar}.

\begin{remark}
The norm $\|.\|_2$ is usually taken as reference   for the $p-$Wasserstein distance. From the equivalence of norms in finite dimensions and as can be done with entropic regularization to make the problem more  tractable numerically,  the $L^p$ norm $\|x_i-y_j\|^p_p=\sum_{\dime=1}^d(x_i^\dime-y_j^\dime)^p$ is here considered for computing the $p-$Wasserstein distance. 
\end{remark}
\vspace*{-0.4cm}

\begin{algorithm}[ht!]
\caption{Estimate Approximate $p$-Wasserstein  $\tilde W_p^p(\mu_x,\mu_y)$ through transshipment}
\label{algo_bar}
\begin{algorithmic}[1]
\Procedure{BarWp}{$x$, $w^x$, $y$, $w^y$, $p$, $\kappa$}
\State $\epsilon=10^{-3}$
\State Initialize $\kappa$ positions $z_k$ randomly from the $m$ and $n$ positions $x_i$ and $y_j$
\Repeat
\State Set $c^{xz}_{ik}=\norm{x_i-z_k}^p$ and $c^{yz}_{jk}=\norm{y_j-z_k}^p$
\State Get $\gamma^x$ and $\gamma^y$ by solving  transshipment problem \eqref{pb:ts} under the constraints \eqref{admiss2}
\State Set $\tilde z_k=z_k$
\State Update positions $z_k$ from $\gamma^x$ and $\gamma^y$ by following Sec. \ref{sec:up_z}
\Until{$\norm{z_k-\tilde z_k}/\norm{\tilde z_k}<\epsilon$}
\State Set $\tilde W=(\langle c^{xz},\gamma^x\rangle^{1/p}+\langle c^{yz},\gamma^y\rangle^{1/p})^p$
\State  \Return $\tilde W$, $\gamma^x$, $\gamma^y$
\EndProcedure
\end{algorithmic}
\end{algorithm}

\subsubsection{Update of positions}\label{sec:up_z}
To  update  positions $z_k\in\R^d$, the problem  \eqref{pb:bar_reg} is solved for fixed transport matrices $\gamma^x$ and $\gamma^y$. For $p\geq 1$, this leads to $d$  convex problems, that can be solved in parallel for each space dimension $\dime$=1\dots d. Dimension indexes $\dime$ are thus omitted  in the following and the problem writes
\begin{equation}\label{pb:z}
\min_{\{z_k\}_{k=1}^\kappa} \sum_ {i=1}^M\sum_{k=1}^\kappa  \gamma^x_{ik} |x_i-z_k|^p+ \sum_ {j=1}^N\sum_{k=1}^\kappa  \gamma^y_{jk} |y_j-z_k|^p,
\end{equation}
Different strategies are considered according to $p$.
Some locations $z_k$ may become useless if $\sum_i\gamma^x_{ik}=\sum_j\gamma^y_{jk}=0$. In this case the corresponding $z_k$ are removed and $\kappa$ is decreased. When considering large scale problems and  few transshipment locations $\kappa$, this almost never happens.\vspace{-0.15cm}
\paragraph{Case $W_2$.}
For $p=2$, there exists  an explicit update formula of the dirac positions to find the unique  minimizer of \eqref{pb:z} with respect to $z$:
\begin{equation}\label{update_z}z_k=\frac{((\gamma^x)^\top x+(\gamma^y)^\top y)_k}{((\gamma^x)^\top \Un_m+(\gamma^y)^\top \Un_n)_k}.\end{equation}
This step acts like the cluster position update in a $\kappa$-mean algorithm. It realizes for each $z_k$ a weighted mean of the positions $x_i$'s and $y_j$'s according to $\gamma^x_{ik}$ and $\gamma^y_{jk}$.\vspace{-0.15cm}

\paragraph{Case $W_p$, $p>2$.}
For $p>2$ the problem \eqref{pb:z}  admits a unique minimizer and is twice differentiable. Newton's method can be considered to approximate the solution:
$$z_k ^{\ell+1}=z_k ^{\ell}- \frac{\left(\left(\gamma^x\otimes |c^{xz^\ell}|^{p-2} \otimes c^{xz^\ell}\right)^\top \Un_m+\left(\gamma^y\otimes |c^{yz^\ell}|^{p-2}\otimes c^{yz^\ell} \right)^\top \Un_n\right)_k}{(p-1)\left((\gamma^x\otimes |c^{xz^\ell}|^{p-2} )^\top \Un_m+(\gamma^y\otimes |c^{yz^\ell}|^{p-2} )^\top \Un_n\right)_k},$$
where $c^{xz^\ell}_{ik}=z^\ell_k-x_i$ and $c^{yz^\ell}_{jk}=z^\ell_k-y_j$, while $\otimes$ denotes the elementwise product between matrices and the power $p-2$ is also element-wise.\vspace{-0.15cm}
\paragraph{Case $W_1$.}
When $p=1$, a global optimum of \eqref{pb:z} can be obtained by taking each $z_k$ as a weighted median of the positions $x_i$ and $y_j$ with respect to the weights $\gamma^x_{ik}$ and $(\gamma^x)_{jk}$.  This operation mainly requires to sort the value of the $x_i$'s and $y_j$'s along each dimension. Notice that alternate minimization on problem \eqref{pb:bar_reg} may  not converge when $p=1$.\vspace{-0.15cm}

\paragraph{Case $W_p$, $p\in]1;2[$.} The problem does not admits a second derivative. An iterative  scheme is then considered by decomposing  $|z_k- x_i|^{p}=|z_k-x_i|^{p-1}|z_k- x_i|$ and solving successive weighted median problems between the positions $x_i$ and $y_j$ with the weights 
$\gamma^x_{ik}|z_k^\ell-x_i|^{p-1}$ and $\gamma^y_{jk}|z_k^\ell-y_j|^{p-1}$.\vspace{-0.14cm}

\subsubsection{Update of transport matrices through Transshipment}\label{sec:up_g}
For fixed $z_k$, problem \eqref{pb:bar_reg} can be solved with classic linear programming optimization tools.
The interesting point is that this barycenter problem is a transshipment problem with $\kappa$ intermediate locations. It can  therefore be formulated in terms of a directed graph with  $m+\kappa+n$ vertices  (i.e. $x_i$, $z_k$ and $y_j$) and  $(m \kappa +\kappa n)$ edges $e_{ik}$ and $e_{kj}$ that correspond to the transport matrices $\gamma^x_{ik}$ and $\gamma^y_{jk}$.
The following cost function is then minimized
\begin{equation}\label{pb:ts}(\gamma^x,(\gamma^y)^\top)\in \uargmin{e=(\{e_{ik}\}, \{e_{kj}\})}\sum_{i=1}^m\sum_{k=1}^\kappa e_{ik}c^{xz}_{ik}+\sum_{k=1}^\kappa\sum_{j=1}^n e_{kj}c^{yz}_{jk},\end{equation}
under the set of constraint $\mathcal{\tilde  P}(w^x,w^y)$ that translates into:
\begin{equation}\label{admiss2}
\left\{\begin{array}{rll}
e_{ik}, e_{kj}&\geq 0&i=1\cdots m,\, k=1\cdots \kappa,\, j=1\cdots n\\
\sum_{k=1}^\kappa e_{ik}&=w^x_i&i=1\cdots m\\
\sum_{j=1}^n e_{kj}-\sum_{i=1}^n e_{ik}&=0&k=1\cdots \kappa\\
-\sum_{j=1}^n e_{kj}&=-w^y_j&j=1\cdots n\\
\end{array}\right.
\end{equation}
This problem can be efficiently solved with the network simplex algorithm \cite{Orlin1993}. An extension of the original non sparse implementation proposed in \cite{Bonneel} is here considered.

\subsubsection{Discussion}

It is well known  \cite{agueh2011barycenters,2015arXiv150707218A} that if $\mu_x$ and $\mu_y$ are respectively supported by $m$ and $n$ dirac masses, then their exists a barycenter  supported by up to $m+n+1$ dirac masses. Hence, the approximation of the barycenter from a set of $\kappa<\min(m,n)$ dirac masses seems to be a bad choice at first sight.
Notice however that,  as pointed out in \cite{Weed}, a barycenter supported by a low dimensional space allows the approximate distance  to be more robust to data outliers.
Next, contrary to entropic regularization of OT, the obtained approximate transport map is here sparse, which allows efficient storage and can be directly used for interpolation purposes without any complex post-processing like sharpening.

\begin{wrapfigure}[20]{r}{7.7cm}%
\begin{center}\vspace{-0.75cm}
\includegraphics[width=\linewidth]{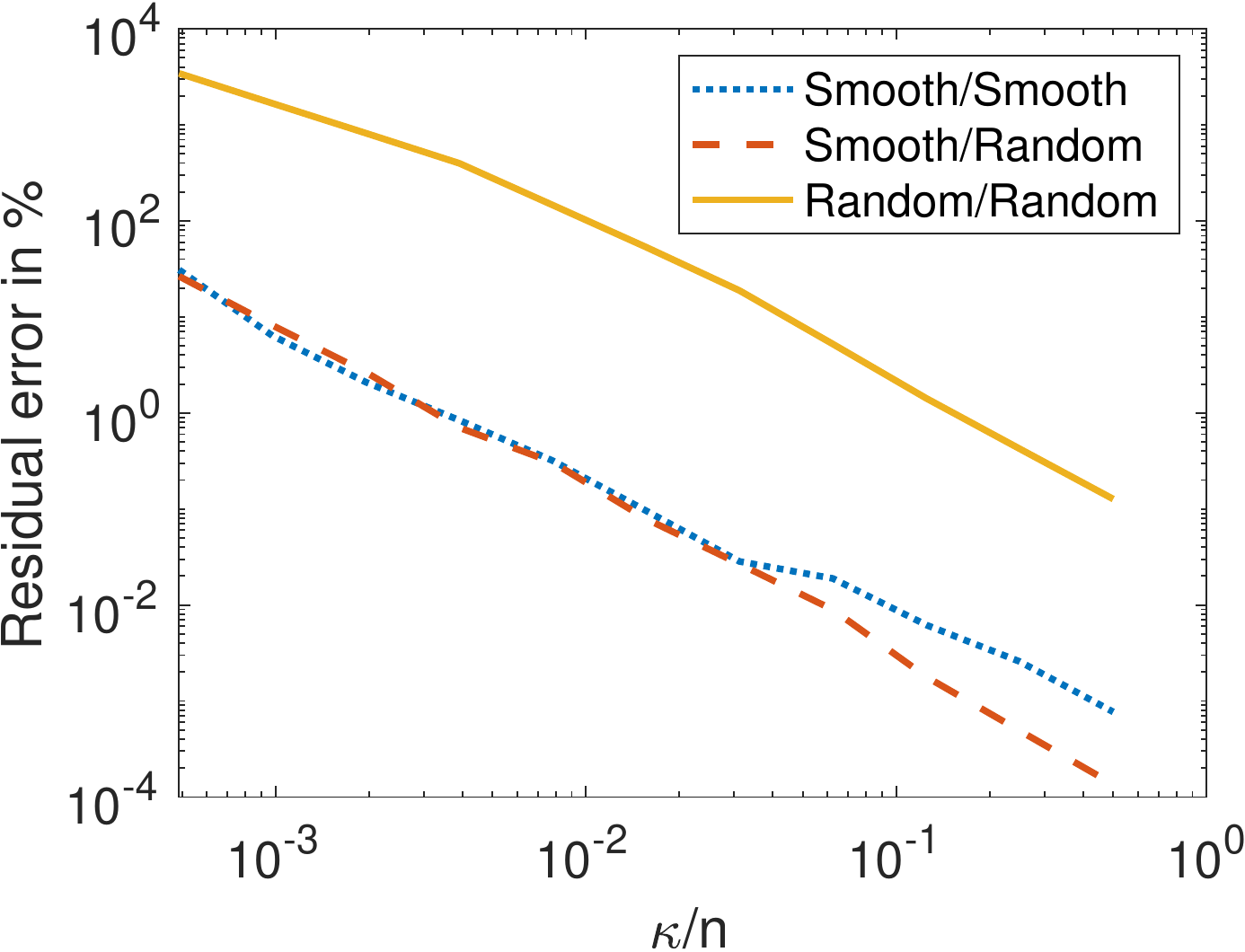}\vspace{-0.5cm}
\caption{\label{fig:cv_data} Accuracy of the approximate Wasserstein distance estimated by solving the transshipment problem \eqref{pb:bar_reg} for increasing values of $\kappa$. Scenarios involving either smooth and/or random data of dimension $m=n=10^5$ have been considered. }
\end{center}
\end{wrapfigure}
Finally, it is worth noting that the approximation depends on the regularity of the data.
As underlined in  \cite{agueh2011barycenters}, when computing the barycenter of  a set of densities $\mu_i$ in the continuous case,  if one of the input data $\mu_i$ is absolutely continuous with respect to the Lebesgue measure, so does the barycenter.
Such observation gives an interesting insight of the  experiments presented in this paper: when at least one of the two data $\mu_x$ or $\mu_y$ is smooth, then the obtained approximate  Wasserstein distance $\tilde W_p$ can be very close to the true one $W_p$  for small values of $\kappa$. Similar behaviour can be observed with semi-discrete optimal transport models where voronoi cells act like barycenters \cite{Merigot}. This point is illustrated in Figure \ref{fig:cv_data} with the comparison of relative errors between true and approximate distances obtained with increasing values of $\kappa$ for different scenarios involving ``random'' or ``smooth'' $2D$ data of dimensions $m=n=10^4$.  When at least one ``regular'' data is involved, then a small relative error $(\tilde W_p-W_p)/W_p<10^{-3}$ (i.e. $0.1\%$) is observed for $\kappa/n=0.015$.

\subsection{Transport refinement and Multi-scale approach}
As illustrated in Figure \ref{fig:cv_data}, the barycenter approach is not sufficient to produce an  accurate  approximation of the $p-$Wasserstein distance between any data. 
First notice that the approximation $\tilde W_p$ of the Wasserstein distance can  be easily improved using the following result that directly considers the transport matrix between $\mu_x$ and $\mu_y$.

\begin{proposition}\label{tot_mat}
Let $\mu_{z}= \sum_{k=1}^\kappa w^z_k\delta_{z_k}$ and  $\gamma^x$ (resp. $\gamma^y$) be an optimal transport matrix from $\mu_x$ (resp. $\mu_y$) to $\mu_z$. 
Let also  $\hat \gamma$ be the transport matrix between $\mu_x$ and $\mu_y$ defined as   $\hat \gamma= \gamma^xD(\gamma^y ) ^\top$, with the rescaling $D^{-1}=\diag(w^z)$ given by the  weights $w^z_k=\sum_i\gamma^x_{ik}=\sum_j\gamma^y_{jk}$. Then the following relation holds
\begin{equation}\label{rell2}\left(\sum_{ij}\hat \gamma_{ij} \|x_i-y_j\|^p\right)^{1/p}\leq W_p(\mu_x,\mu_{z})+W_p(\mu_y,\mu_{z}).
\end{equation} 
\end{proposition}
\vspace*{0.2cm}

\noindent
\begin{proof}
The matrix $\hat\gamma_{ij}=\sum_k \gamma^x_{ik}\gamma^y_{jk}/w^z_k$ is an admissible  transport matrix between $\mu_x$ and $\mu_y$, since $\sum_j \hat \gamma_{ij}=\sum_k\gamma^x_{ik}=w^x_i$ and $\sum_i \hat \gamma_{ij}=\sum_k\gamma^y_{jk}=w^y_j$. Following the proof of Minkowski's inequality:
\begin{equation*}
\begin{split}
&\underset{ij}{\Sigma}\hat \gamma_{ij} \|x_i-y_j\|^p=\underset{ijk}{\Sigma} \gamma^x_{ik}\gamma^y_{jk}/w^z_k   \|x_i-y_j\|^{p-1}\|x_i-y_j\|\\
\leq&\underset{ijk}{\Sigma} (\gamma^x_{ik}\gamma^y_{jk}/w^z_k)^{(p-1)/p+1/p} \|x_i-y_j\|^{p-1}(\|x_i-z_k\|+\|y_j-z_k\|)\\
\leq&\left(\underset{ijk}{\Sigma}\gamma^x_{ik}\gamma^y_{jk}/w^z_k \|x_i-y_j\|^{p} \right)^{(p-1)/p} \left( \left( \underset{ijk}{\Sigma}\gamma^x_{ik}\gamma^y_{jk}/w^z_k\|x_i-z_k\|^p\right)^{1/p}\hspace{-0.2cm}+ \left( \underset{ijk}{\Sigma}\gamma^x_{ik}\gamma^y_{jk}/w^z_k \|y_j-z_k\|^p\right)^{1/p}\right)\\
\leq&\left(\underset{ij}{\Sigma}\tilde\gamma_{ij} \|x_i-y_j\|^{p} \right)^{(p-1)/p} \left( \left(\underset{ik}{\Sigma}\gamma^x_{ik}\|x_i-z_k\|^p\right)^{1/p}+ \left( \underset{jk}{\Sigma}\gamma^y_{jk} \|y_j-z_k\|^p\right)^{1/p}\right),
\end{split}
\end{equation*}
then  gives \eqref{rell2}. Notice that the proof of the triangle inequality \cite{Clement} is here also obtained,  since $W_p(\mu_x,\mu_y)\leq(\underset{ij}{\Sigma}\tilde \gamma_{ij} \|x_i-y_j\|^p)^{1/p}$.
\end{proof}
\vspace*{0.5cm}

\noindent
From this proposition, the approximation $\hat  W_p(\mu_x,\mu_y)$ corresponding to a barycenter $\mu_{1/2} =\sum_{k=1}^\kappa w^z_k\delta_{z_k}$  solution of \eqref{pb:bar_reg} with  transport matrices $\gamma^x$ and $\gamma^y$ is  defined as
\begin{equation}\label{rell3}W_p(\mu_x,\mu_y)\leq\hat  W_p(\mu_x,\mu_y)= \langle \gamma^x\diag(w^z)^{-1}(\gamma^y ) ^\top ,c^{xy}\rangle^{1/p}\leq\tilde  W_p(\mu_x,\mu_y).\end{equation}
 For matching or interpolation purposes, $\hat\gamma= \gamma^x(\diag(w^z))^{-1}(\gamma^y ) ^\top$ gives  a sparse
 approximation of the optimal transport matrix. 
However, as illustrated in Figure \ref{fig:block}, such approach maps all locations $x_i$ and $y_j$ that transit by $z_k$, resulting in a poor block transport matrix for small values of $\kappa$. 

\begin{figure}[ht!]
\begin{center}
\begin{tabular}{cccccc}
\includegraphics[width=0.12\linewidth,height=0.12\linewidth]{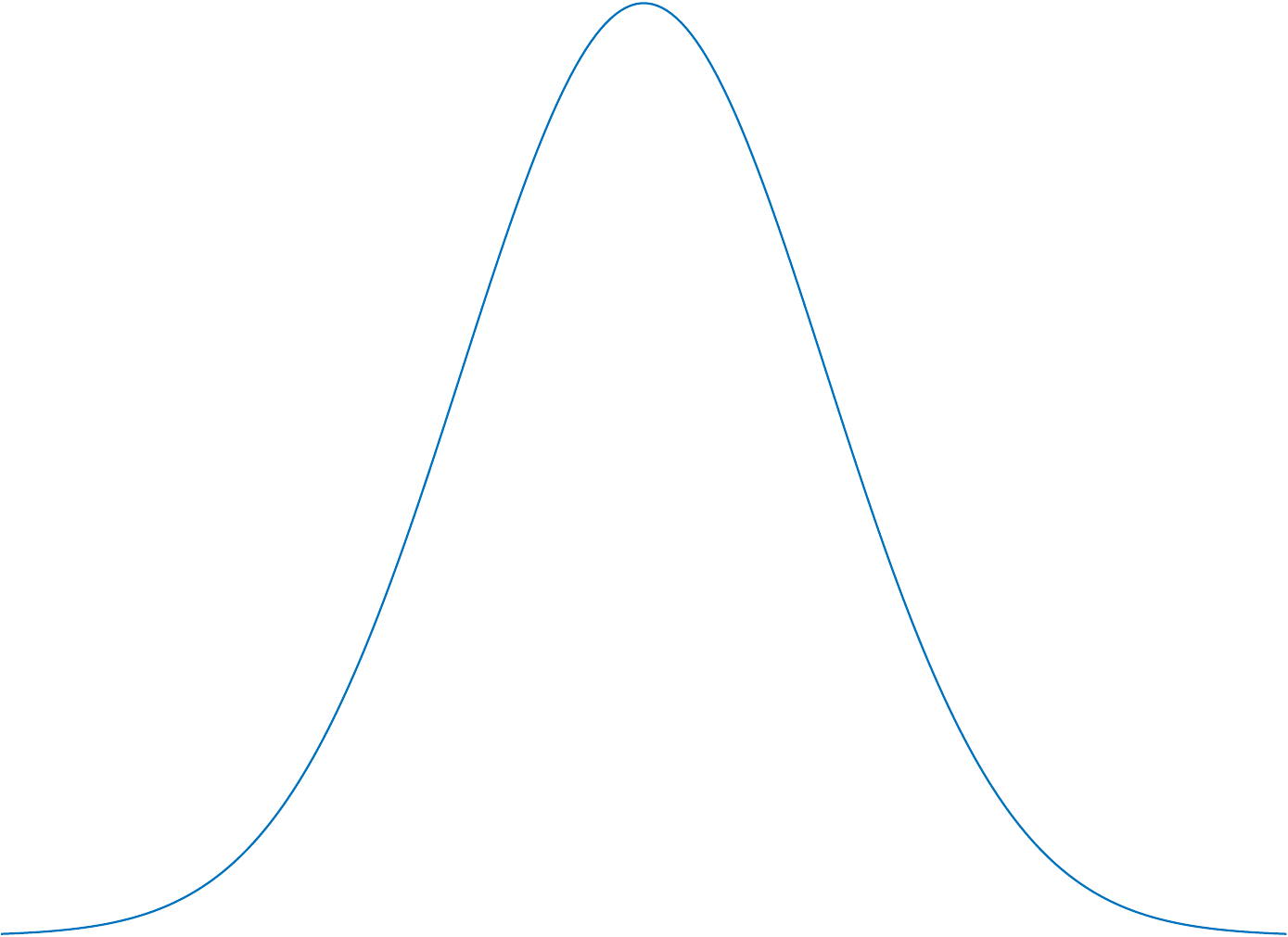}&\includegraphics[width=0.12\linewidth,height=0.12\linewidth]{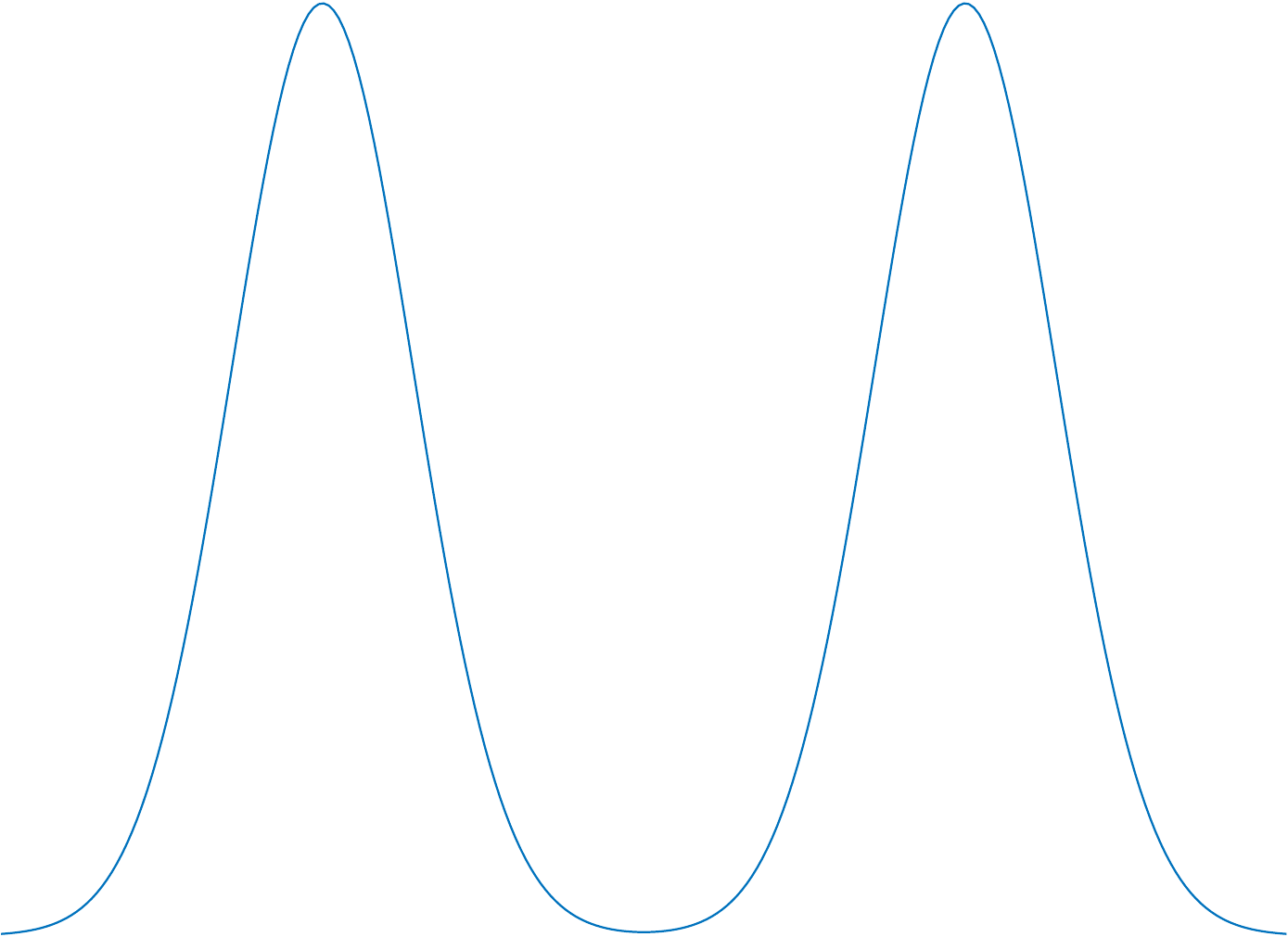}&\includegraphics[height=0.15\linewidth]{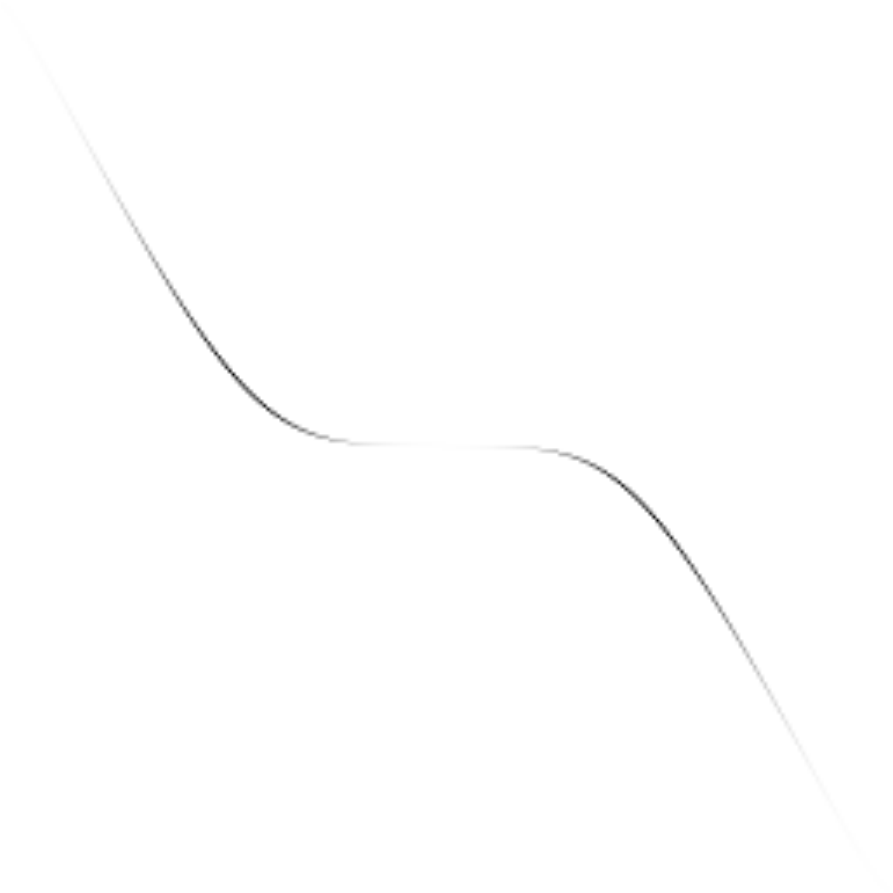}&\includegraphics[height=0.15\linewidth]{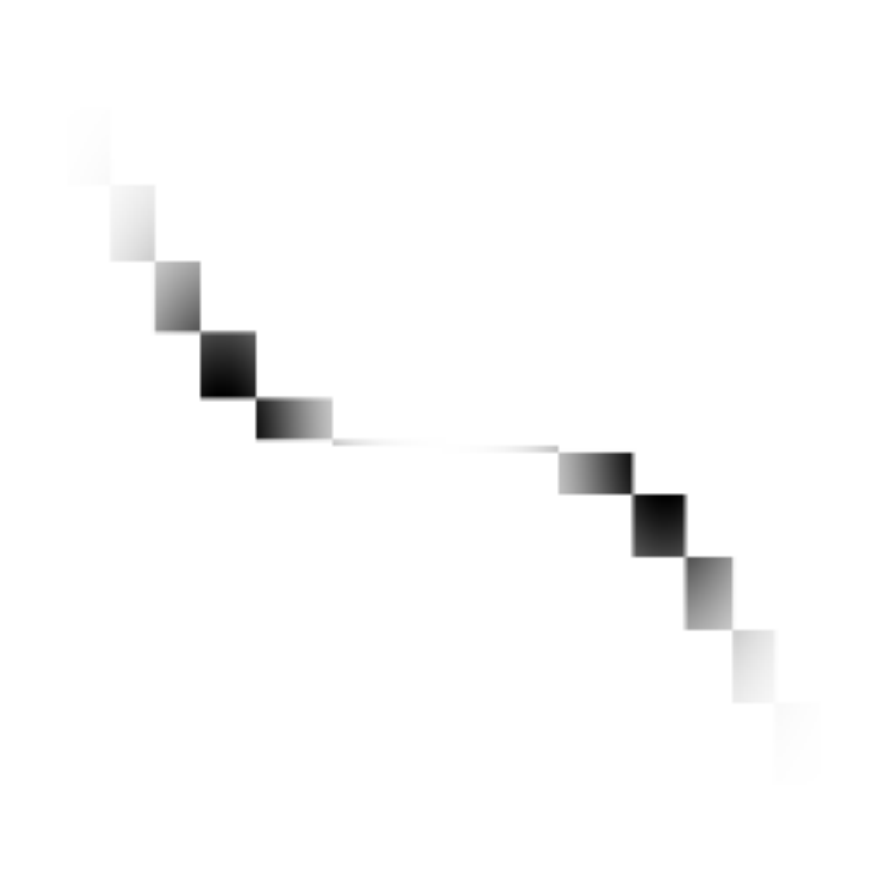}&\includegraphics[height=0.15\linewidth]{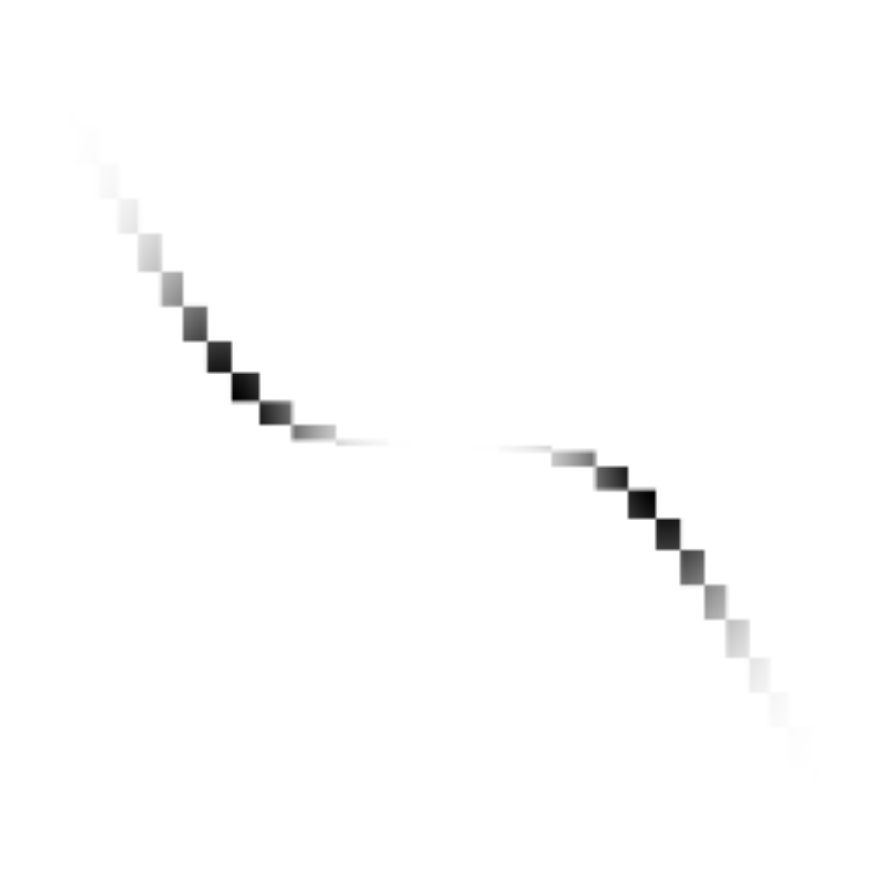}&\includegraphics[height=0.15\linewidth]{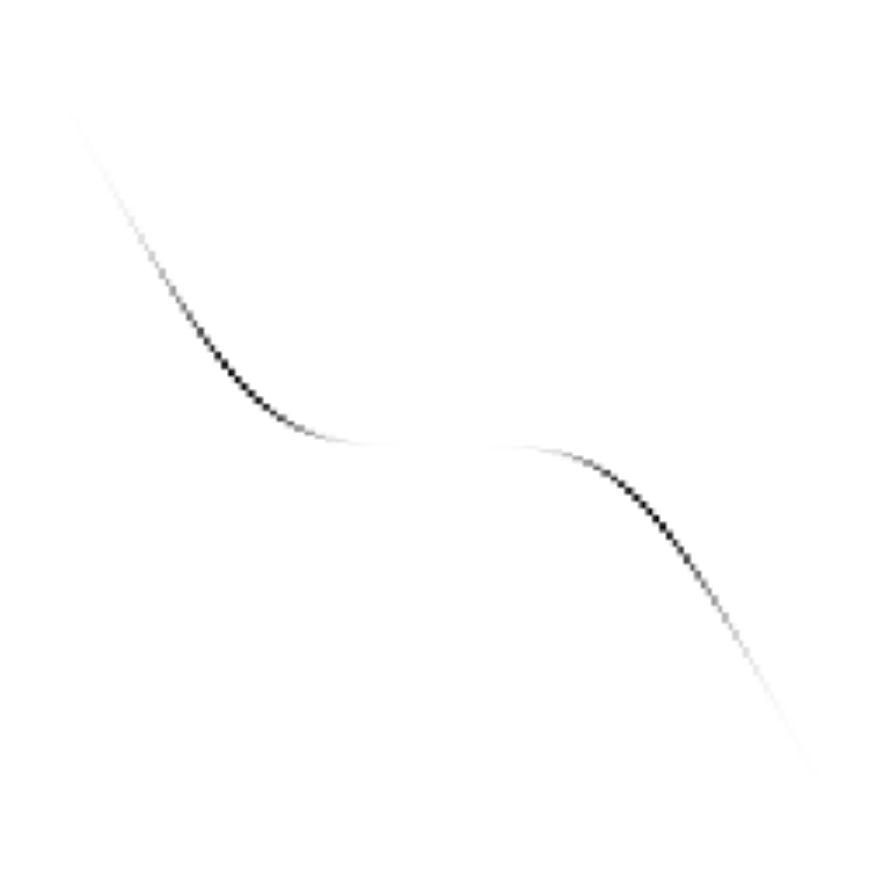}\\
$\mu_x$&$\mu_y$&True transport&$\kappa=16$&$\kappa=32$&$\kappa=128$
\end{tabular}
\caption{\label{fig:block} Illustration of the block transport matrices $ \hat \gamma$ estimated from data $\mu_x$ and $\mu_y$ through  transshipment  for increasing values of $\kappa$.}
\end{center}
\end{figure}

In order to get a sparser and more accurate approximation of  the optimal transport matrix and associated distance, it is necessary to ``untie'' the links between locations passing through $z_k$.
To do so, a solution is to consider $\kappa$ optimal transportation sub-problems, by refining the transport of the mass transshipped through $z_k$. The barycenter approach then acts as a clustering and for each intermediate location $z_k$, the $p-$Wasserstein distance between the following partial discrete densities is computed:
\begin{equation}
\begin{split}
\mu_x^k=\sum_{i=1}^m\gamma^x_{ik}\delta_{x_i}&\hspace{1cm}
\mu_y^k=\sum_{j=1}^n\gamma^y_{jk}\delta_{y_j}, 
\end{split}
\end{equation}
where the number of active dimensions are expected to be reduced: $m_k=\#\{\gamma^x_{ik}>0\}<m$ and $n_k=\#\{\gamma^y_{jk}>0\}<n$.
The multi-scale approach is then performed as follows.
The Wasserstein distance $W_p(\mu_x^k,\mu_y^k)$ is estimated exactly with network simplex \cite{Bonneel} if the sum of $m_k$ and $n_k$ is small enough. Otherwise the barycenter approach is recursively applied to the subproblem.  From numerical experiments, the threshold $n_k+m_k<2000$ has been chosen to reach the best compromise between numerical accuracy and computational cost. The whole process is illustrated in Figure \ref{fig:trans}  and detailed in Algorithm \ref{algo}.

\begin{figure}[ht!]
\begin{center}
\begin{tabular}{ccc}
\includegraphics[width=0.31\linewidth]{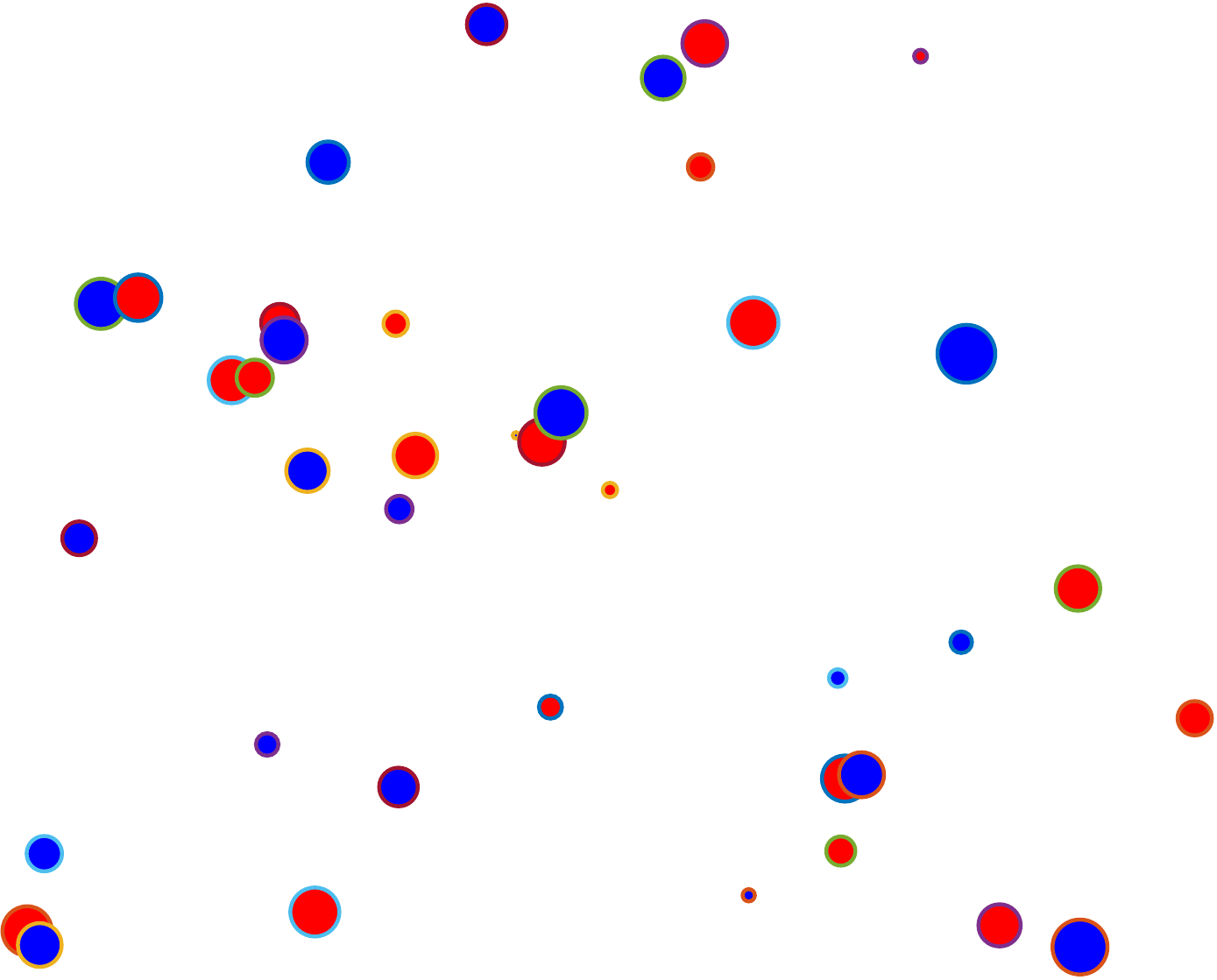}&\hspace{0.2cm}\includegraphics[width=0.31\linewidth]{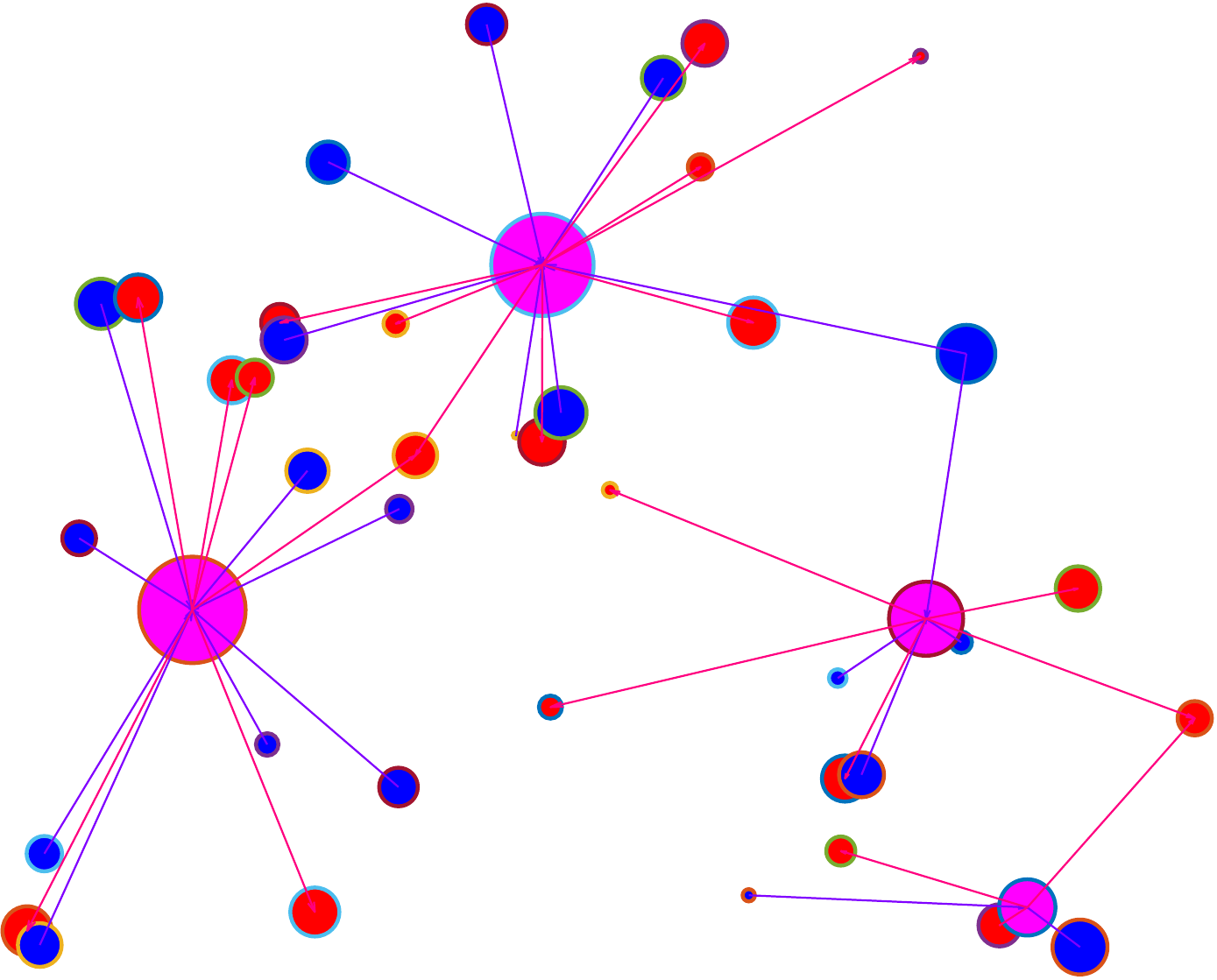}&\hspace{0.2cm}\includegraphics[width=0.31\linewidth]{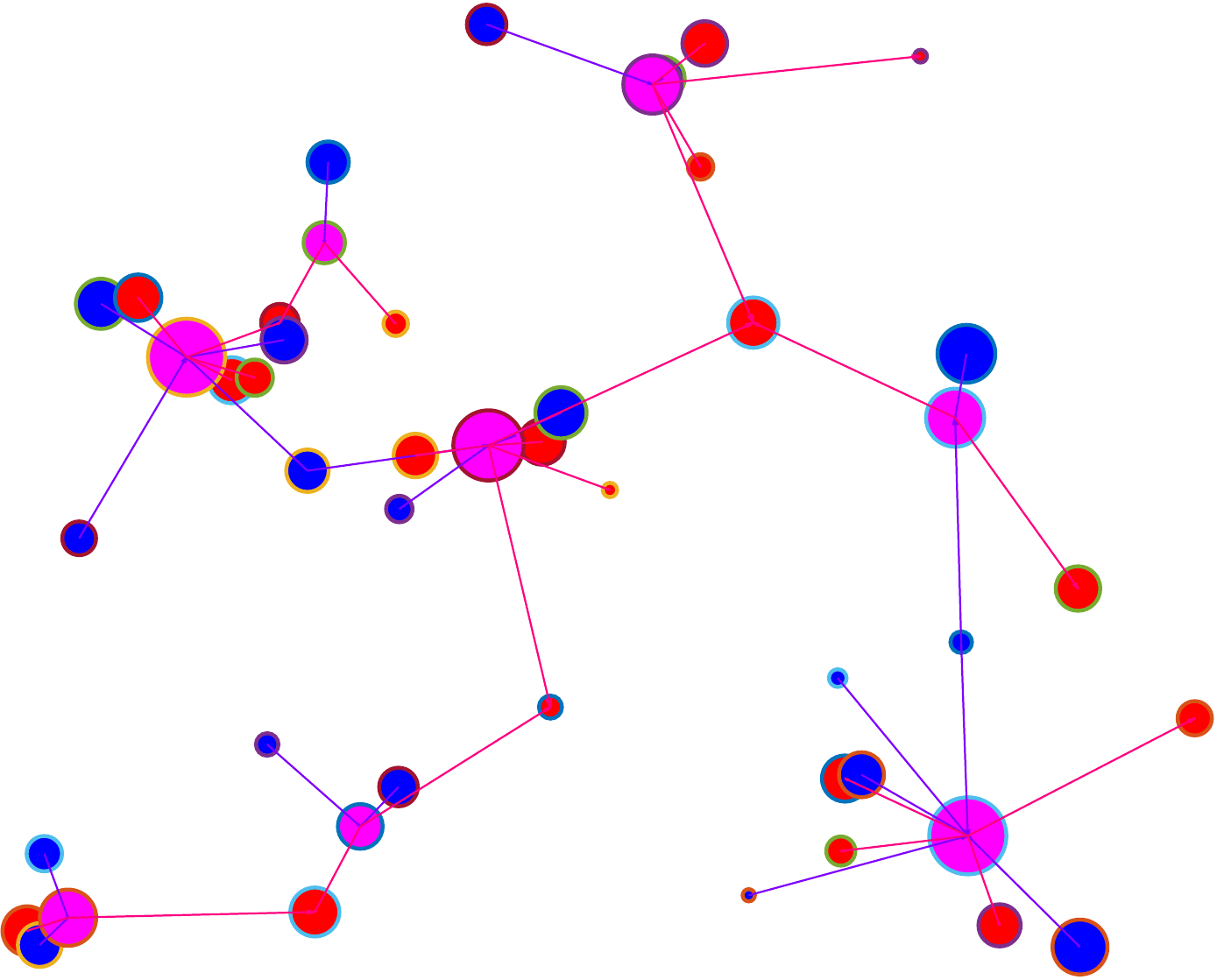}\\
Data&\hspace{0.2cm}Transshipment&\hspace{0.2cm}Transshipment\\
&\hspace{0.2cm}$\kappa=4$&\hspace{0.2cm}$\kappa=8$\vspace{0.4cm}\\
\includegraphics[width=0.31\linewidth]{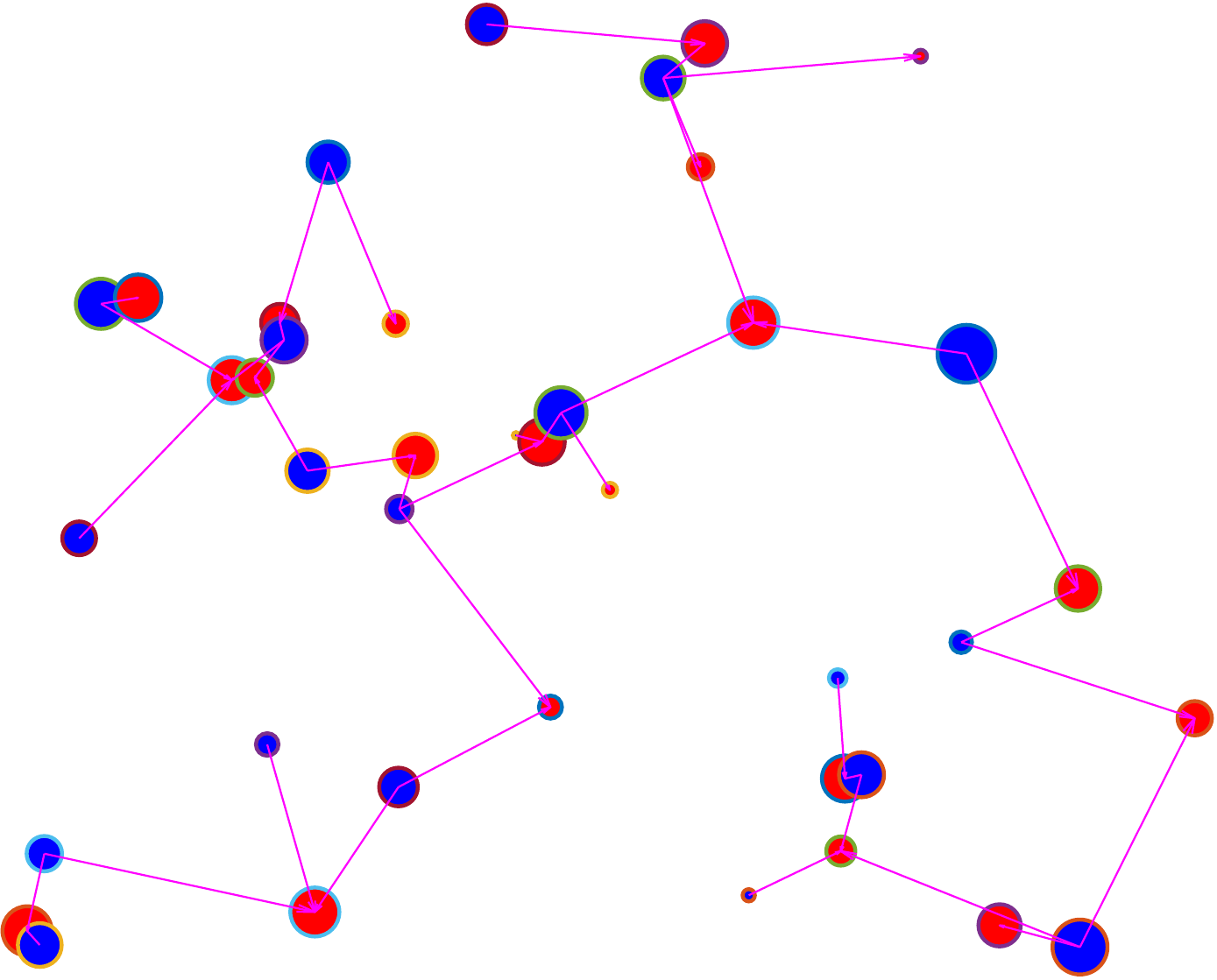}&\includegraphics[width=0.31\linewidth]{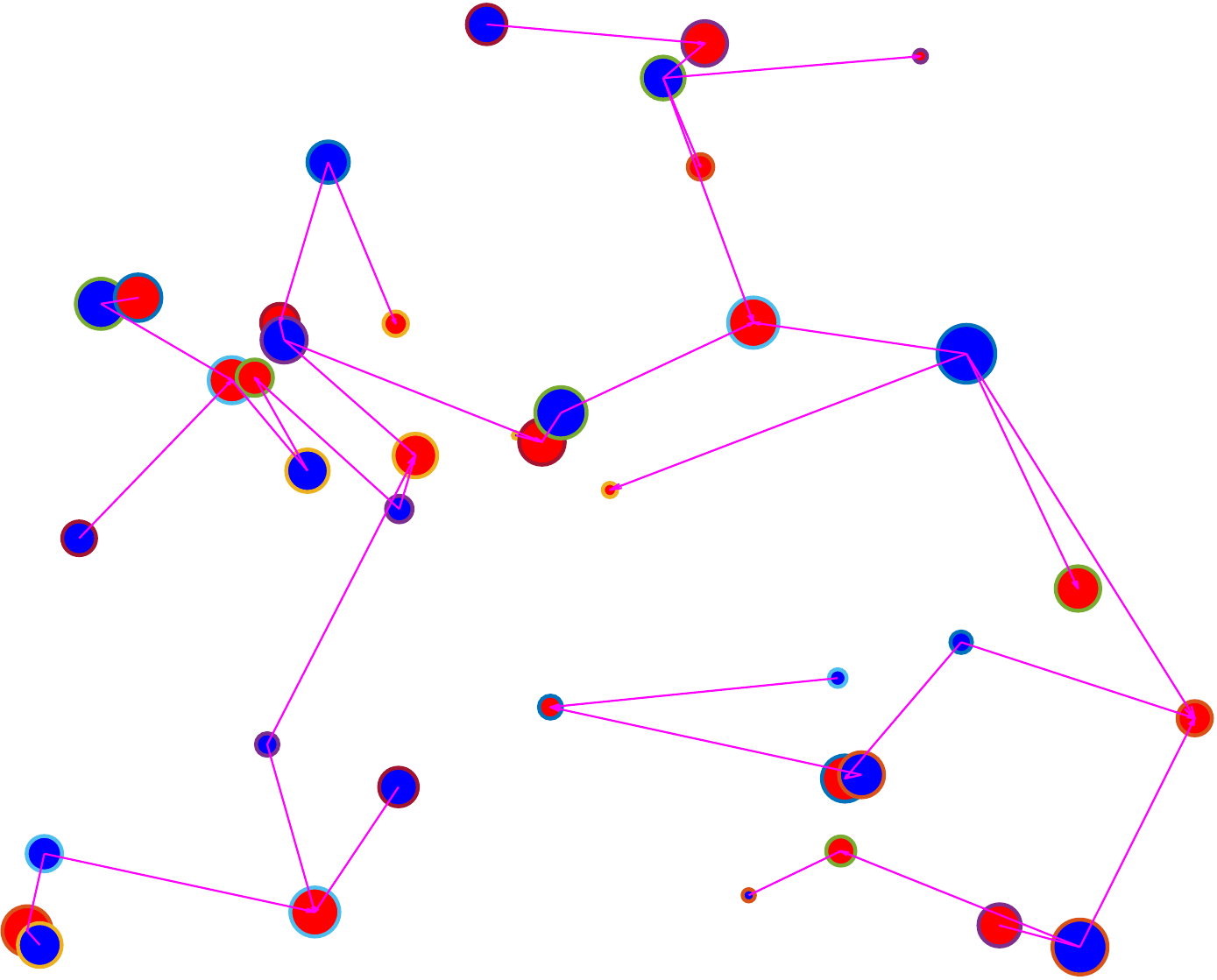}&\includegraphics[width=0.31\linewidth]{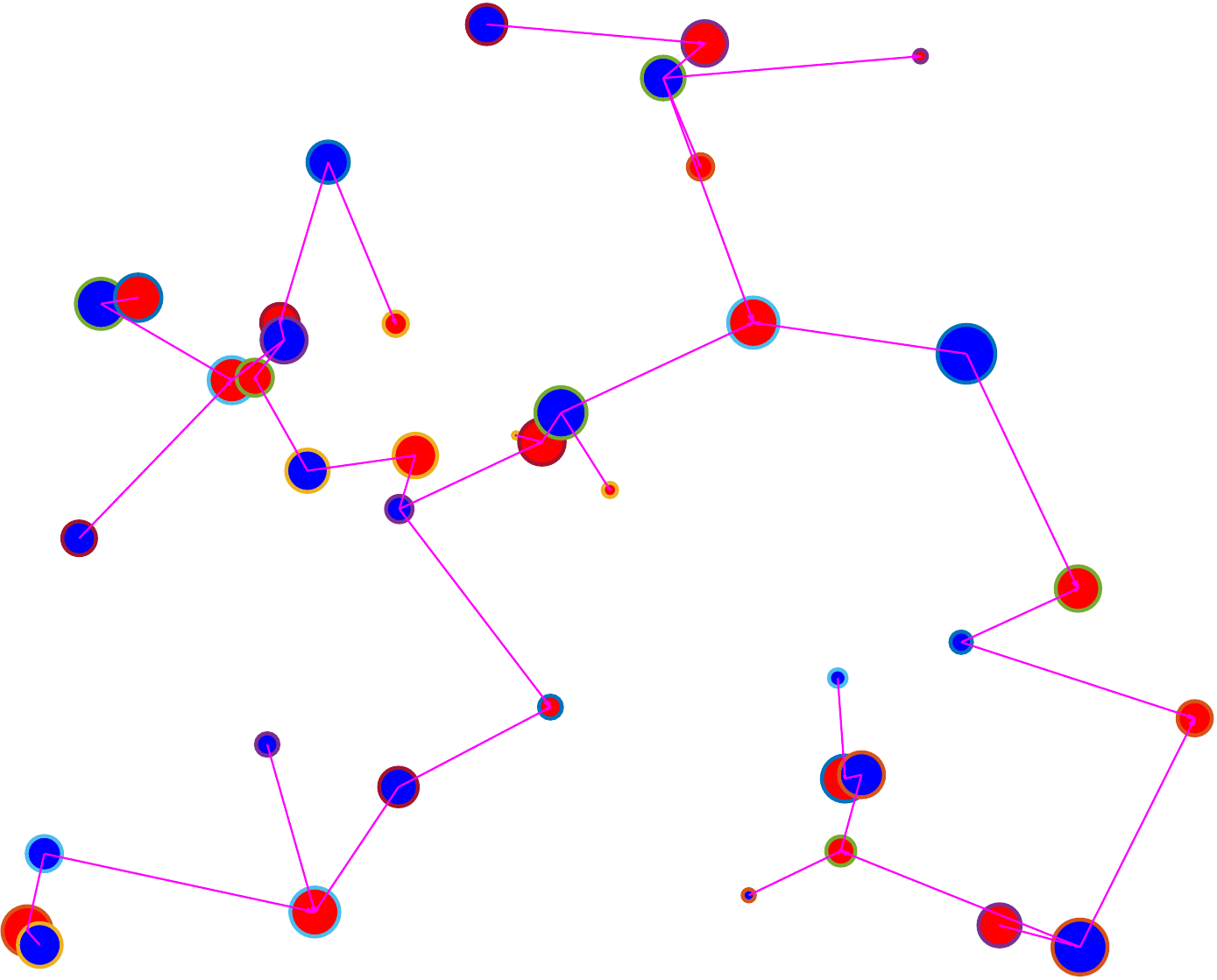}\\
Optimal transportation&\hspace{0.2cm}Refined transportation&\hspace{0.2cm}Refined transportation\\
&\hspace{0.2cm}$\kappa=4$&\hspace{0.2cm}$\kappa=8$
\end{tabular}
\caption{\label{fig:trans} Illustration of transportation obtained from refinement of transshipment with $\kappa=4$  and $\kappa=8$ intermediate locations. edges indicate there is a mass transport (i.e.  $\gamma_{ij}>0$) between locations $x_i$ (in blue) and $y_j$ (in red). Optimal Transportation is here recovered when refining the transshipment solution  obtained with $\kappa=8$.}
\end{center}
\end{figure}

\begin{algorithm}[ht!]
\caption{Multi-scale estimation of approximate $p$-Wasserstein distance $\hat W_p^p(\mu_x,\mu_y)$}
\label{algo}
\begin{algorithmic}[1]
\Procedure{ApproxWp}{$x$, $w^x$, $y$, $w^y$, $p$, $\kappa$}
\State ($\tilde W$, $\gamma^x,\gamma^y)$=\textsc{BarWp} ({$x$, $w^x$, $y$, $w^y$, $p$, $\kappa$})\Comment{Clustering with Algo. \ref{algo_bar}}
\State Initialize $\hat W=0$, $\hat \gamma={\bf 0}_{m\times n}$
\For {$k=1$ to $\kappa$}\Comment{Can be done in parallel}
\State  Set $\mathcal{I}^k=\{i|\, \gamma^x_{ik}>0\}$, $\mathcal{J}^k=\{j|\, \gamma^y_{jk}>0\}$
\State Set $m_k=\#\mathcal{I}^k$, $n_k=\#\mathcal{J}^k$
\State Set $x^k=x_{\mathcal{I}^k}$, $y^k=y_{\mathcal{J}^k}$, $w^{x^k}= \gamma^x_{\mathcal{I}^kk}$, $w^{y^k}=\gamma^y_{\mathcal{J}^kk}$
\If{$m_k+n_k<T$} 
\State $(W^k,\gamma^k)=$\textsc{Wp}($x^k$, $w^{x^k}$, $y^k$, $w^{y^k}$, $p$)\Comment{Compute exact $W_p^p$ with Algo. \ref{algoW}}
\Else 
\State $(W^k,\gamma^k)=$\textsc{ApproxWp}($x^k$, $w^{x^k}$, $y^k$, $w^{y^k}$, $p$, $\kappa$)\Comment{Approximation of $W_p^p$}
\EndIf  
\State $\hat W=\hat W+W^k$
\State $\hat\gamma_{\mathcal{I}^k\mathcal{J}^k}=\gamma^k$
\EndFor
\State  \Return $\hat W$, $\hat\gamma$
\EndProcedure
\end{algorithmic}
\end{algorithm}

\section{Experiments}
True distance $2-$Wasserstein distances \eqref{def:Wp_discrete} are here compared with the approximated ones computed with the multi-scale procedure of Algorithm \ref{algo}. Exact and approximate distances are respectively obtained with the \CC network simplex implementation of \cite{Bonneel} based on the  graph library LEMON \cite{Lemon} and the proposed   transshipment extension \footnote{The code is available at \url{https://www.math.u-bordeaux.fr/~npapadak/GOTMI/codes.php}.}.  The experiments have been realized on  a standard Macbook with a processor Intel Core i7 2,2 GHz and 16 Go of RAM. 
\paragraph{Accuracy}

In order to study the performance of the proposed approximation $\hat W_p$, the  $32\times 32$ and $64\times 64$ and $128\times 128$ images of the  Benchmark \cite{bench} have been considered.  For each image size, this data set contains $10$ classes of different densities, and each class contains $10$ images. The exact and approximate distances have been computed between all (i.e. $\approx 5000$) possibles pairs of images. This has been done for different values of $\kappa$ and a threshold of $T=2000$ in Algorithm 3.
For each experiment, the mean and median relative errors between approximate  and exact methods are computed and presented in Table \ref{mean_res_bench}. 

\begin{table}[ht!]\begin{center}
\begin{tabular}{c|c|c|c|c|c}
&\multicolumn{2}{c|}{$\kappa=4$}&\multicolumn{2}{c|}{$\kappa=16$}\\
&Mean &Median&Mean &Median\\
\hline
$n=32\times 32=1024$&$2.70\%$&$2.18\%$&$1.61\%$&$0.90\%$\\
\hline
$n=64\times 64=4096$&$3.56\%$ & $2.61\%$&$1.31\%$ & $0.81\%$\\
\hline
$n=128\times 128=16384$&$3.49\%$&$2.34\%$&$1.38\%$&$0.82\%$\\
\hline
\end{tabular}
\caption{\label{mean_res_bench} Mean and median relative errors between approximate and true EMD  on the Benchmark \cite{bench} for  $32\times 32$, $64\times 64$ and $128\times 128$ images and different values of $\kappa$.}
\end{center}
\end{table}

A more detailed presentation is then given in Figure \ref{inter_class_bench}, with the mean relative errors over  intra and inter classes experiments in the case of $128\times 128$ images. As can be observed, significant errors are obtained when unstructured random data (classes 4 and 10) are involved. In all other cases, the relative errors are very low.

\begin{figure}[ht!]
\begin{center}
{
\setlength\tabcolsep{1.5pt}
\def\arraystretch{1.15}
\begin{tabular}{c|c|c|c|c|c|c|c|c|c|c|}&\includegraphics[width=0.06\linewidth]{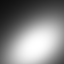}&\includegraphics[width=0.06\linewidth]{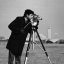}&\includegraphics[width=0.06\linewidth]{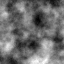}&\includegraphics[width=0.06\linewidth]{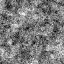}&\includegraphics[width=0.06\linewidth]{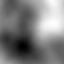}&\includegraphics[width=0.06\linewidth]{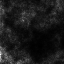}&\includegraphics[width=0.06\linewidth]{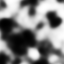}&\includegraphics[width=0.06\linewidth]{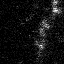}&\includegraphics[width=0.06\linewidth]{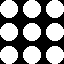}&\includegraphics[width=0.06\linewidth]{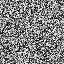}\vspace{-0.1cm}\\
\hline 
{\raisebox{-.42\height}{\includegraphics[width=0.06\linewidth]{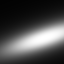}}}&\cellcolor{gray!3}$0.4\%$&\cellcolor{gray!2}$0.4\%$&\cellcolor{gray!3}$0.4\%$&\cellcolor{gray!3}$0.4\%$&\cellcolor{gray!2}$0.3\%$&\cellcolor{gray!2}$0.3\%$&\cellcolor{gray!2}$0.4\%$&\cellcolor{gray!3}$0.5\%$&\cellcolor{gray!3}$0.5\%$&\cellcolor{gray!2}$0.4\%$\\
\hline 
{\raisebox{-.42\height}{\includegraphics[width=0.06\linewidth]{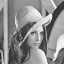}}}&\cellcolor{gray!2}$0.4\%$&\cellcolor{gray!8}$1.3\%$&\cellcolor{gray!5}$0.9\%$&\cellcolor{gray!5}$0.8\%$&\cellcolor{gray!4}$0.7\%$&\cellcolor{gray!4}$0.6\%$&\cellcolor{gray!8}$1.4\%$&\cellcolor{gray!7}$1.2\%$&\cellcolor{gray!7}$1.2\%$&\cellcolor{gray!4}$0.6\%$\\
\hline 
{\raisebox{-.42\height}{\includegraphics[width=0.06\linewidth]{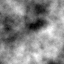}}}&\cellcolor{gray!3}$0.4\%$&\cellcolor{gray!5}$0.9\%$&\cellcolor{gray!8}$1.2\%$&\cellcolor{gray!10}$1.6\%$&\cellcolor{gray!4}$0.6\%$&\cellcolor{gray!3}$0.5\%$&\cellcolor{gray!9}$1.5\%$&\cellcolor{gray!7}$1.1\%$&\cellcolor{gray!8}$1.4\%$&\cellcolor{gray!6}$1.0\%$\\
\hline 
{\raisebox{-.42\height}{\includegraphics[width=0.06\linewidth]{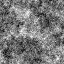}}}&\cellcolor{gray!3}$0.4\%$&\cellcolor{gray!5}$0.8\%$&\cellcolor{gray!10}$1.6\%$&\cellcolor{gray!17}$2.7\%$&\cellcolor{gray!5}$0.8\%$&\cellcolor{gray!4}$0.6\%$&\cellcolor{gray!12}$2.0\%$&\cellcolor{gray!7}$1.1\%$&\cellcolor{gray!10}$1.6\%$&\cellcolor{gray!24}$3.8\%$\\
\hline 
{\raisebox{-.42\height}{\includegraphics[width=0.06\linewidth]{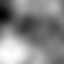}}}&\cellcolor{gray!2}$0.3\%$&\cellcolor{gray!4}$0.7\%$&\cellcolor{gray!4}$0.6\%$&\cellcolor{gray!5}$0.8\%$&\cellcolor{gray!3}$0.4\%$&\cellcolor{gray!4}$0.6\%$&\cellcolor{gray!2}$0.4\%$&\cellcolor{gray!6}$1.0\%$&\cellcolor{gray!3}$0.4\%$&\cellcolor{gray!1}$0.2\%$\\
\hline 
{\raisebox{-.42\height}{\includegraphics[width=0.06\linewidth]{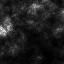}}}&\cellcolor{gray!2}$0.3\%$&\cellcolor{gray!4}$0.6\%$&\cellcolor{gray!3}$0.5\%$&\cellcolor{gray!4}$0.6\%$&\cellcolor{gray!4}$0.6\%$&\cellcolor{gray!4}$0.6\%$&\cellcolor{gray!11}$1.8\%$&\cellcolor{gray!7}$1.1\%$&\cellcolor{gray!9}$1.5\%$&\cellcolor{gray!7}$1.2\%$\\
\hline 
{\raisebox{-.42\height}{\includegraphics[width=0.06\linewidth]{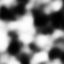}}}&\cellcolor{gray!2}$0.4\%$&\cellcolor{gray!8}$1.4\%$&\cellcolor{gray!9}$1.5\%$&\cellcolor{gray!12}$2.0\%$&\cellcolor{gray!2}$0.4\%$&\cellcolor{gray!11}$1.8\%$&\cellcolor{gray!11}$1.8\%$&\cellcolor{gray!9}$1.5\%$&\cellcolor{gray!9}$1.5\%$&\cellcolor{gray!11}$1.7\%$\\
\hline 
{\raisebox{-.42\height}{\includegraphics[width=0.06\linewidth]{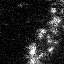}}}&\cellcolor{gray!3}$0.5\%$&\cellcolor{gray!7}$1.2\%$&\cellcolor{gray!7}$1.1\%$&\cellcolor{gray!7}$1.1\%$&\cellcolor{gray!6}$1.0\%$&\cellcolor{gray!7}$1.1\%$&\cellcolor{gray!9}$1.5\%$&\cellcolor{gray!5}$0.7\%$&\cellcolor{gray!7}$1.1\%$&\cellcolor{gray!7}$1.2\%$\\
\hline 
{\raisebox{-.42\height}{\includegraphics[width=0.06\linewidth]{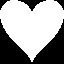}}}&\cellcolor{gray!3}$0.5\%$&\cellcolor{gray!7}$1.2\%$&\cellcolor{gray!8}$1.4\%$&\cellcolor{gray!10}$1.6\%$&\cellcolor{gray!3}$0.4\%$&\cellcolor{gray!9}$1.5\%$&\cellcolor{gray!9}$1.5\%$&\cellcolor{gray!7}$1.1\%$&\cellcolor{gray!13}$2.1\%$&\cellcolor{gray!37}$6.1\%$\\
\hline 
{\raisebox{-.42\height}{\includegraphics[width=0.06\linewidth]{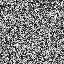}}}&\cellcolor{gray!2}$0.4\%$&\cellcolor{gray!4}$0.6\%$&\cellcolor{gray!6}$1.0\%$&\cellcolor{gray!24}$3.8\%$&\cellcolor{gray!1}$0.2\%$&\cellcolor{gray!7}$1.2\%$&\cellcolor{gray!11}$1.7\%$&\cellcolor{gray!7}$1.2\%$&\cellcolor{gray!37}$6.1\%$&\cellcolor{gray!60}$9.7\%$\\
\hline 
\end{tabular}

}
\caption{\label{inter_class_bench} Detailed mean relative errors for intra and inter classes tests on images of size $128\times 128$.\vspace{-0.4cm}}
\end{center}
\end{figure}

\paragraph{Computational cost}

In Figure \ref{compare_speed}, the running time for computing an approximate $2-$Wasserstein distance, with Algorithm \ref{algo} and $\kappa=16$, is compared  with the \CC network simplex implementation and its sparse multi-threaded extension proposed in  \cite{Bonneel}. The running times  become  asymptotically very interesting with the multi-threaded extension, but due to memory storage, these methods can not handle dimensions $n$  larger than $3.10^4$ on the  considered computer.
As transsshipment involves problems of size $n\kappa$, it can be applied to  data containing more dirac masses  and thus deal with one additional  order of magnitude  ($n=2.10^5$). With the proposed  full \CC implementation of the transshipment problem, the $\kappa$ sub-problems are solved successively. The provided Matlab interface calling the \CC code through a simple parfor loop  with $4$ workers is thus much faster. 
Optimal transshipment matrices being dense for small values of $\kappa$, it is counter-productive to consider sparse optimized implementation in the multi-scale framework.

\begin{figure}[ht!]
\begin{center}
\includegraphics[width=9.4cm]{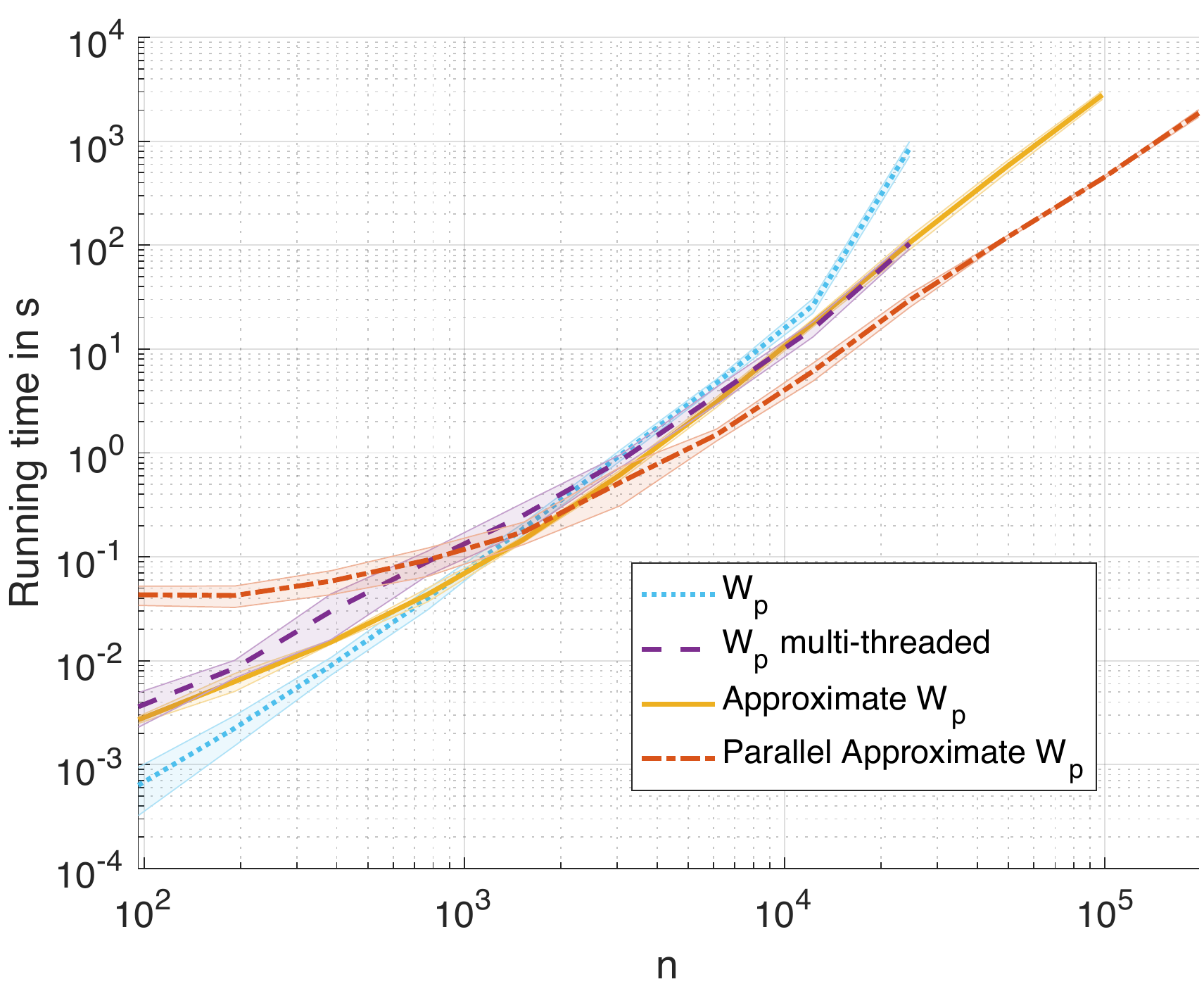}\vspace{-0.1cm}
\caption{\label{compare_speed}Comparison of running times for computing $2-$Wassersetin distances for different values of $n$: \CC network simplex \cite{Bonneel}, its multi-threaded extension  and the proposed multi-scale approximation (matlab \CC mex), that is the only one being able to handle high scale problems.\vspace{-0.25cm}}
\end{center}
\end{figure}
\section{Conclusion}
This paper presents an empirical method for approximating Wasserstein distances. It is based on existing concepts used in parallel works \cite{Weed,2018arXiv180507416A,2019arXiv190108949P}. The contribution is to provide an efficient multi-scale implementation  able to deal with unstructured point clouds while providing sparse transport matrices. Numerical experiments demonstrate  the accuracy of the computed approximate distance, while the involved computational cost are improved with respect to the literature.
As a perspective, several intermediary transhipment levels could be considered, in relation to branched transport \cite{bernot2008optimal}
It would also be of interest to add constraints encouraging an homogeneous repartition of the number of points transiting by each location $z_k$ of the barycenter, or at least a more uniform distribution of the barycenter weights $w^z_k$. The $\kappa$ multi-scale sub-problems that can be solved in parallel would  have similar dimensions, and theoretical guarantees on the overall running time could be given.

\paragraph{Acknowledgements}
This study has been carried out with financial support from the French
State, managed by the French National Research Agency (ANR GOTMI) (ANR-16-CE33-0010-01).
The project has also received funding from the European Union’s Horizon 2020 research and innovation programme under the Marie Skłodowska-Curie grant agreement No 777826.\vspace{-0.1cm}
\bibliographystyle{abbrv}

\end{document}